\documentclass[Threeside,11pt]{article}
\usepackage{etoolbox}
\makeatletter
\patchcmd{\thebibliography}
{\list}
{\small          
	\setstretch{1.2}
	\list}
{}{}
\makeatother
\usepackage{setspace}
\usepackage{tikz}
\usepackage{float}
\definecolor{myPurple}{RGB}{109,90,207}
\definecolor{myGreen}{RGB}{4,130,67}
\definecolor{myGold}{RGB}{253,177,71}
\definecolor{myBurgundy}{RGB}{63,1,44}
\definecolor{myTeal}{RGB}{8,119,127}
\definecolor{myCream}{RGB}{255,216,177}
\definecolor{myOrange}{RGB}{225,119,1}

\usetikzlibrary{arrows.meta,bending,calc}
\tikzset{
	ball3Donly/.style={
		circle, minimum size=1.4cm,
		shading=ball,
		ball color=myPurple,
		draw=myPurple
	}
}
\pgfmathsetmacro{\wNear}{2.0pt}
\pgfmathsetmacro{\wMid}{1.3pt}
\pgfmathsetmacro{\wFar}{0.7pt}
\pgfmathsetmacro{\bend}{6mm}
\usepackage [latin1]{inputenc}
\usepackage[english]{babel}
\usepackage{amsmath}
\usepackage{amsthm}
\usepackage{algorithm}
\usepackage{cite}
\usepackage{amsfonts}
\usepackage{amssymb}
\usepackage{mathrsfs}
\usepackage{graphicx}
\usepackage{colortbl,dcolumn}
\usepackage{marvosym}
\usepackage{ifsym}
\usepackage{xcolor}
\usepackage{enumitem}
\allowdisplaybreaks
\usepackage[
colorlinks,
citecolor=blue,
linkcolor=blue]{hyperref} 
       \newtheorem{lemma}{\bf Lemma}[section]
       \newtheorem{theorem}{\bf Theorem}[section]

       \newtheorem{definition}{\bf Definition}[section]
       \newtheorem{remark}{\bf Remark}[section]

       \numberwithin{equation}{section}

\begin{document}
\title{{\sl An infinite-dimensional Kolmogorov theorem and the construction of almost periodic breathers}}

\author{Zhicheng Tong\thanks{School of Mathematics, Jilin University, Changchun 130012, P. R.  China. Email: \url{tongzc25@jlu.edu.cn}} 
	\and 
	Yong Li\thanks{Corresponding author. School of Mathematics, Jilin University, Changchun 130012, P. R.  China;   Center for Mathematics and Interdisciplinary Sciences,  Northeast  Normal  University, Changchun 130024, P. R. China. Email: \url{liyong@jlu.edu.cn}}
}

\date{}
\maketitle

\begin{abstract}
In this paper, we present two infinite-dimensional Kolmogorov theorems based on non-resonant frequencies of Bourgain's Diophantine type or even weaker conditions. To be more precise, under a Legendre-type nondegeneracy condition for an infinite-dimensional Hamiltonian system, we prove the persistence of a full-dimensional KAM torus with a universally prescribed frequency independent of any spectral asymptotics. As an application, we prove that for a class of perturbed networks with weakly coupled oscillators described by
\[\frac{{{{\rm d}^2}{x_n}}}{{{\rm d}{t^2}}} + V'\left( {{x_n}} \right) = \varepsilon_n  {W'\left( {{x_{n + 1}} - {x_n}} \right) - \varepsilon_{n-1}W'\left( {{x_n} - {x_{n - 1}}} \right)} ,\quad n \in \mathbb{Z},\]
or even for more general perturbed networks, frequency-preserving almost periodic breathers do persist, provided that the local potential $ V $ and the coupling potential $ W $ satisfy certain assumptions. In particular, this yields the first frequency-preserving result for the Aubry--MacKay conjecture \cite{MR1304442,MR1353962}.
\vspace{2mm}
\end{abstract}

\tableofcontents

\section{Introduction}\label{SEC1}
\setcounter{footnote}{0}
\renewcommand{\thefootnote}{{\arabic{footnote}}}
This paper presents an infinite-dimensional version of the  Kolmogorov theorem, demonstrating the persistence of a full-dimensional invariant torus in an infinite-dimensional perturbed Hamiltonian system under a Legendre-type nondegeneracy condition. The surviving torus preserves its original frequency of Bourgain's Diophantine type without requiring any spectral asymptotics. Using this result, we make a novel contribution to the Aubry--MacKay conjecture  \cite{MR1304442,MR1353962}. We prove that, for certain weakly coupled infinite-dimensional networks (lattice systems), frequency-preserving almost periodic breathers do persist under perturbations. For the reader's convenience, we first summarize our two main results below. Detailed quantitative formulations and historical context for the second result will be presented in subsequent sections.\vspace{2mm}

\noindent\textbf{[Main Result I]} \textbf{(Theorems \ref{FULLT1},  \ref{FULLT2} and \ref{FULLT3}: Infinite-dimensional Kolmogorov theorem)} Consider an infinite-dimensional, analytic, nondegenerate, and non-integrable Hamiltonian system. For frequencies of Bourgain's Diophantine type (Definition \ref{FULLDio}) or even weaker (Definition \ref{weakdio}), the unperturbed full-dimensional tori with such frequencies survive small perturbations. They undergo only slight deformations while preserving their original frequencies.
\vspace{3mm}

\noindent\textbf{[Main Result II]} \textbf{(Theorems \ref{MACDINGLI} and \ref{TH22}: Persistence of frequency-preserving almost periodic breathers)} 
For a class of perturbed networks with weakly coupled oscillators:
\begin{equation}\notag
	\frac{{{{\rm d}^2}{x_n}}}{{{\rm d}{t^2}}} + V'\left( {{x_n}} \right) = \varepsilon_n  {W'\left( {{x_{n + 1}} - {x_n}} \right) - \varepsilon_{n-1}W'\left( {{x_n} - {x_{n - 1}}} \right)} ,\quad n \in \mathbb{Z},
\end{equation}
or even for more general perturbed networks:
\[	\frac{{{{\rm d}^2}{x_n}}}{{{\rm d}{t^2}}} + V'\left( x_n \right) = \sum\limits_{p =  - {m_{1,n}}}^{ - 1} {{\varepsilon _{n,p}}{W_{n,p}^\prime} \left( {{x_{n + p}} - {x_n}} \right)}  - \sum\limits_{p=1}^{{m_{2,n}}} {{\varepsilon _{n,p}}{W_{n,p}^\prime} \left( {{x_{n + p}} - {x_n}} \right)} ,\quad n \in \mathbb{Z},\]
frequency-preserving almost periodic breathers do persist, provided that these systems satisfy certain assumptions. Moreover, these frequencies are of Bourgain's Diophantine type (see Definition \ref{FULLDio}).
\vspace{2mm}

Next, we briefly outline the \textbf{primary motivations} and our \textbf{main contributions}. It is well known that constructing full-dimensional invariant tori for infinite-dimensional Hamiltonian systems is a challenging problem. Existing results are almost exclusively of Arnold type, which allow for frequency-drifting and focus primarily on measure estimates. From the viewpoint of dynamical systems, a central question is the extent to which the original dynamics persists under small perturbations. In this context, the persistence of invariant tori is of fundamental importance. Equally important is the \textit{preservation of frequencies}. Consequently, in finite dimensions, the Kolmogorov  theorem concerning  frequency preservation has been extensively studied. In stark contrast, rigorous statements of such results appear to be entirely absent in the infinite-dimensional setting. It is worth noting that, for a frequency prescribed universally (in the measure-theoretic sense), it is \textit{preferable} to use  the infinite-dimensional Diophantine condition introduced by Bourgain \cite{MR2180074};  see Corsi \textit{et al.} \cite{MR4781767}   for details.   Thus, achieving such a Kolmogorov-type result with Bourgain's non-resonance condition (or even weaker) is of independent interest, which is one of the primary motivations for this paper. On the other hand, from the perspective of physical applications, many nonlinear models are inherently nondegenerate. Thus, intuitively, one might  expect that these systems may preserve more of the unperturbed dynamics under perturbation, such as breathers and their frequencies. This is indeed a \textit{novel} finding and constitutes another primary motivation for our work. This was previously unknown perhaps due to the lack of an appropriate infinite-dimensional Kolmogorov theorem based on Bourgain's Diophantine non-resonance condition. Furthermore, when the frequency is preserved, there is \textit{no need} to assume any spectral asymptotics (e.g., as in Bourgain  \cite{MR2180074}  and Kuksin--P\"oschel  \cite{MR1370761}). Instead, it is only necessary to ensure certain  well-definedness, which is highly desirable from an applied perspective, at least in \textit{networks (lattice systems)}.

It is worth noting that preserving the prescribed non-resonant frequency associated with the full-dimensional torus is a challenging task that necessitates certain nondegeneracy. Otherwise, the prescribed frequency drifts. In contrast to the truncation method commonly used in KAM theory, we employ the generating function approach based on the nondegeneracy of the Hamiltonian system. This ensures that the frequency remains unchanged throughout the KAM iteration process, in the spirit of Kolmogorov \cite{MR0068687} and Salamon \cite{MR2111297}.
Furthermore, achieving a balance among the spatial structure, non-resonance condition, and regularity presents additional challenges in validating the KAM scheme.

The remainder of this paper is organized as follows. In Section \ref{SEC2}, we first provide some basic characterizations of the spatial structure, non-resonance, and regularity. Based on these, we present two theorems corresponding to Main Result I: Theorem \ref{FULLT1} on Bourgain's Diophantine non-resonance condition and Theorem  \ref{FULLT2} on the weak Diophantine non-resonance condition. Subsequently, we provide a discussion of the spatial structure and present a more general Theorem \ref{FULLT3} without proof. In Section \ref{SECMACKAY}, we utilize  Main Result I to prove Main Result II, namely the persistence of frequency-preserving almost periodic breathers for certain weakly coupled networks. Then, for the reader's convenience, a more detailed comparison of our main results with pertinent previous literature is presented in Section \ref{SECNEW}. Finally, the proofs of Theorems \ref{FULLT1} and \ref{FULLT2} are postponed to Section \ref{SEC4}, utilizing an improved infinite-dimensional KAM technique.

\section{Persistence of  full-dimensional frequency-preserving tori}\label{SEC2}

\subsection{Infinite-dimensional non-resonance condition, spatial structure and norms}
Before presenting  our KAM results, let us first recall the definition of the 	infinite-dimensional  Diophantine frequency introduced by Bourgain \cite{MR2180074}. For more details,  see, e.g., Biasco \textit{et al.} \cite{MR4091501} and Montalto--Procesi \cite{MR4201442}.
\begin{definition}[Infinite-dimensional Diophantine frequency]\label{FULLDio}
	For $ 0<\gamma<1 $ and $ \mu > 1 $, the infinite-dimensional Diophantine non-resonance condition for $ \omega  \in {{\mathbb{R}}^\mathbb{Z}} $ is defined as
	\begin{equation}\notag 
		\left| {\omega  \cdot \ell } \right| > \gamma \prod\limits_{j \in \mathbb{Z}} {\frac{1}{{\left( {1 + {{\left| {{\ell _j}} \right|}^\mu }{{\left\langle j \right\rangle }^\mu }} \right)}}} ,\quad  \text{$ \forall \ell  \in {\mathbb{Z}^\mathbb{Z}} $ with $ 0 < \sum\limits_{j \in \mathbb{Z}} {\left| {{\ell _j}} \right|}  <  + \infty $,} 
	\end{equation}
	where  $ \left\langle j \right\rangle : = \max \left\{ {1,\left| j \right|} \right\} $ for $ j \in \mathbb{Z} $.
\end{definition}

To show the universality of the  infinite-dimensional Diophantine non-resonance condition, let us consider the Diophantine set $ {\mathcal{D}_{\gamma ,\mu }} $ defined by 
\[{\mathcal{D}_{\gamma ,\mu }}: = \left\{ {\omega  \in {{\left[ {1,2} \right]}^\mathbb{Z}}:\quad \text{$ \omega $ is an infinite-dimensional Diophantine frequency as in Definition \ref{FULLDio}}} \right\}.\]
Fortunately, as pointed out by Bourgain \cite{MR2180074} and also by Biasco \textit{et al.} \cite{MR4091501}, one has the following estimate: there exists a positive constant $ C(\mu) $  depending only on $ \mu $, such that 
\[\mathbb{P}\left( {{{\left[ {1,2} \right]}^\mathbb{Z}}\setminus {\mathcal{D}_{\gamma ,\mu }}} \right) \leqslant C\left( \mu  \right)\gamma= \mathcal{O}\left( \gamma  \right),\quad \text{as $  \gamma \to 0^+ $,}\]
where $\mathbb{P}$ is the product measure on $[1,2]^\mathbb{Z}$.
Therefore, the Diophantine non-resonance condition is indeed \textit{universal} (or \textit{typical}) in a measure-theoretical sense, i.e., $ \mathbb{P} ( {{{ [ {1,2}  ]}^\mathbb{Z}}\setminus \bigcup\nolimits_{0 < \gamma  < 1} {{\mathcal{D}_{\gamma ,\mu }}} }  ) = 0 $.  Nevertheless, throughout the present paper, we do not require the  range restriction $ {{\left[ {1,2} \right]}^\mathbb{Z}} $ of $\omega$, and we may simply assume $\omega\in \mathbb{R}^\mathbb{Z}$ (also in the definition of   $ {\mathcal{D}_{\gamma ,\mu }} $)\footnote{In this context, when referring to a measure of frequencies rather than an individual frequency, one considers probability measures supported on bounded subsets of $\mathbb{R}^{\mathbb{Z}}$.}.

We now introduce some basic notation for the relevant function spaces.
 	For $ \sigma>0 $ and $ \eta \geqslant 2 $,  define the infinite-dimensional \textit{thickened  torus} as
\begin{equation}\label{FUUHUAN}
	\mathbb{T}_\sigma ^\infty : = \left\{ {x = {{({x_j})}_{j \in \mathbb{Z}}}:\quad {x_j} \in \mathbb{C},\;\operatorname{Re} {x_j} \in \mathbb{T},\;\left| {\operatorname{Im} {x_j}} \right| \leqslant \sigma {{\left\langle j \right\rangle }^\eta },\; j \in \mathbb{Z}} \right\}.
\end{equation}
Based on the thickened torus $ \mathbb{T}_\sigma ^\infty $, we define the space
of analytic functions $u: \mathbb{T}_\sigma ^\infty  \to \mathbb{C} $ as 
\begin{equation}\label{FUUJIEXI}
	\mathcal{G}\left( {\mathbb{T}_\sigma ^\infty } \right): = \left\{ {u (x) = \sum\limits_{\ell  \in \mathbb{Z}_ * ^\infty } {\widehat u_{\ell}{{\rm e}^{{\rm i}\left\langle {\ell ,x} \right\rangle  }}} :\quad \|u\|_{\sigma }: = \sum\limits_{\ell  \in \mathbb{Z}_ * ^\infty } {\left| {\widehat u_{\ell}} \right|{{\rm e}^{\sigma {{\left| \ell  \right|}_\eta }}}}  <  + \infty } \right\},
\end{equation}
endowed with the set of infinite integer vectors with finite support: 
\[\mathbb{Z}_ * ^\infty : = \left\{ {\ell  \in {\mathbb{Z}^\mathbb{Z}}:\quad {{\left| \ell  \right|}_\eta }: = \sum\limits_{j \in \mathbb{Z}} {{{\left\langle j \right\rangle }^\eta }\left| {{\ell _j}} \right|}  <  + \infty } \right\}.\]
Observe that for $ u \in {\mathcal{G}}\left( {\mathbb{T}_\sigma ^\infty } \right) $, the radius of analyticity of each angular $ x_j $ increases as $ |j| \to +\infty $.  Such a spatial structure  effectively ensures  the feasibility of Fourier analysis in the infinite-dimensional context.
We also denote  by $ {\mathcal{G}_0}\left( {\mathbb{T}_\sigma ^\infty } \right) $ the space of analytic functions with vanishing zeroth Fourier constants:
\[{\mathcal{G}_0}\left( {\mathbb{T}_\sigma ^\infty } \right): = \left\{ {u:\quad u \in \mathcal{G}\left( {\mathbb{T}_\sigma ^\infty } \right),\;\widehat u_0 = 0} \right\}.\]
Moreover, for $ \varsigma>0 $, denote by
\[{\mathscr{D}_r}: = \left\{ {{{\left( {{y_j}} \right)}_{j \in \mathbb{Z}}}:\quad  y_j \in \mathbb{C}, \;{\left\| y \right\|^ *_\varsigma }: = \sum\limits_{j \in \mathbb{Z}} {\left| {{y_j}} \right|{{\left\langle j \right\rangle }^{\varsigma}}}<r } \right\}\]
the complex neighborhood of $ 0 \in {\mathbb{C}^\mathbb{Z}} $, and denote by
\[{\mathscr{D}_{\sigma ,r}}: = \mathbb{T}_\sigma ^\infty  \times {\mathscr{D}_r} \subseteq \left({\mathbb{C}^\mathbb{Z}}/(2\pi {\mathbb{Z}^\mathbb{Z}})\right) \times {\mathbb{C}^\mathbb{Z}}\]
the complex neighborhood of $ {\mathbb{T}^\infty } \times \left\{ 0 \right\} $. Hereafter, we identify ${\mathbb{T}^\infty}$ with ${\mathbb{T}^\mathbb{Z}}$. For an analytic function $ u = u\left( {x,y} \right) $  on $ \mathscr{D}_{\sigma ,r} $, we define  its weighted norm as
\[\left\|u\left( {x,y} \right)\right\|_{{\sigma ,r}}: = \sum\limits_{\ell  \in \mathbb{Z}_ * ^\infty } {\mathop {\sup }\limits_{y \in {\mathscr{D}_{r }}} \left| {{u_\ell }} \right|{{\rm e}^{\sigma {{\left| \ell  \right|}_\eta }}}} ,\]
where $u_\ell = \widehat u_\ell(y)$ are the Fourier coefficients of $u$ with respect to $x$ and are functions of $y$.
Similarly,	for a matrix valued  function $ \mathscr{A}\left( {x,y} \right) = {\left( {{\mathscr{A}^{\left( {i,j} \right)}}\left( {x,y} \right)} \right)_{i,j \in \mathbb{Z}}} $ with $ \mathscr{A}^{\left( {i,j} \right)}\left( {x,y} \right) \in \mathbb{C} $ on $ \mathscr{D}_{\sigma,r} $, we define its weighted norms as
\begin{equation}\notag 
	\left\|\mathscr{A}\left( {x,y} \right)\right\|_{\sigma }: = \mathop {\sup }\limits_{i,j\in \mathbb{Z}} \left\|{\mathscr{A}^{\left( {i,j} \right)}}\left( {x,y} \right)\right\|_{\sigma },\quad \left\|\mathscr{A}\left( {x,y} \right)\right\|_{{\sigma ,r}}: = \mathop {\sup }\limits_{i,j\in \mathbb{Z}} \left\|\mathop {\sup }\limits_{y \in \mathscr{D}_{r}} \left| {{\mathscr{A}^{\left( {i,j} \right)}}\left( {x,y} \right)} \right|\right\|_{\sigma }.
\end{equation}
Let $\mathscr{I}_\sigma^\infty$ denote the set of all invertible operators $\mathscr{B}: \mathbb{T}_\sigma^\infty \to \mathbb{T}_\sigma^\infty$, where the norm on $\mathbb{T}_\sigma^\infty$ is defined as in \eqref{FUUJIEXI}. For simplicity, we write the operator norm of $\mathscr{B}$ as $\|\mathscr{B}\|_{\mathscr{I}_\sigma^\infty}:={\sup _{{{\left\| u \right\|}_\sigma } = 1}}{\left\| {{\mathscr{B}}u} \right\|_\sigma }$. 
In the subsequent analysis, we denote $  \langle { \cdot , \cdot }  \rangle $ as the standard inner product in the infinite-dimensional setting. It can be readily verified that any instance of $  \langle { \cdot , \cdot }  \rangle $ that appears in this paper is well-defined under the specific spatial structure considered.

\subsection{KAM with  infinite-dimensional Diophantine frequency-preserving}
With the previous preparation, we are now able to present for infinite-dimensional Hamiltonian systems the following KAM persistence theorem regarding full-dimensional frequency-preserving  invariant tori.

\begin{theorem}[KAM via the infinite-dimensional Diophantine non-resonance condition]\label{FULLT1}
	Let $ 0<\sigma<1$ and $ \eta \geqslant 2 $. Assume that the  frequency $ \omega  \in {{\mathbb{R}}^\mathbb{Z}} $ satisfies the infinite-dimensional Diophantine condition in Definition \ref{FULLDio}. Additionally, 	assume that $ \mathscr{H}(x,y) $ is a real analytic Hamiltonian function defined on $ {\mathscr{D}_{\sigma ,\sigma }} $, which is $ 1 $-periodic in the variables $ (x_j)_{j \in \mathbb{Z}} $, and satisfies
	\begin{align}
		\label{FULL1}\left\|\mathscr{H}\left( {x,0} \right) - \int_{{\mathbb{T}^\infty }} {\mathscr{H}\left( {\xi ,0} \right){\rm{d}}\xi } \right\|_{\sigma } &\leqslant {{\rm e}^{ - K}},\\
		\label{FULL2}\left\|{\mathscr{H}_y}\left( {x,0} \right) - \omega \right\|_{\sigma } &\leqslant {{\rm e}^{ - K{{ {\sigma } }^{ - \frac{2}{\eta} }}}},\\
		\label{FULL3}\left\|{\mathscr{H}_{yy}}\left( {x,y} \right) - \mathscr{Q}\left( {x,y} \right)\right\|_{{\sigma ,\sigma }} &\leqslant \sigma^{-1}{{\rm e}^{ - K}},
	\end{align}
	where $ K > 0 $ is a sufficiently large constant independent of $ \sigma $, and $ \mathscr{Q}\left( {x,y} \right) \in {\mathbb{C}^{\mathbb{Z} \times \mathbb{Z}}} $ is a Hermitian  and analytic matrix-valued function on $ {{\mathscr{D}}_{\sigma ,\sigma }} $ satisfying
	\begin{equation}\label{FULLninininini}
		\left\|\mathscr{Q}\left( {x,y} \right)\right\|_{{\sigma ,\sigma }} \leqslant M,
		\quad \left\|{\left( {\int_{{\mathbb{T}^\infty }} {\mathscr{Q}\left( {x,0} \right){\rm d}x} } \right)^{ - 1}}\right\|_{{\mathscr{I}_\sigma^\infty}} \leqslant M,
	\end{equation}
	with $ M>0 $ independent of $ \sigma $\footnote{The second condition in \eqref{FULLninininini} is commonly referred to as the Legendre-type nondegeneracy condition in the finite-dimensional setting; hence, we adopt this terminology in the infinite-dimensional case as well.}.
	
	 Under these assumptions,  there exists a real analytic symplectic transformation $ z = \phi \left( \zeta  \right) $ of the form
	\begin{equation}\notag
		z = \left( {x,y} \right),\quad \zeta  = \left( {\xi ,\kappa } \right),\quad x = u\left( \xi  \right),\quad y = v\left( \xi  \right) + \left\langle u_\xi {\left( \xi  \right)^{ - 1}},\kappa \right\rangle,
	\end{equation}
	mapping $ {\mathscr{D}_{\sigma /4,\sigma /4}} $ into $ {\mathscr{D}_{\sigma ,\sigma }} $, such that $ u\left( \xi  \right) - {\rm id}  $ and $ v\left( \xi  \right) $ are  $1$-periodic in all variables, and the Hamiltonian function $ \mathscr{W} = \mathscr{H} \circ \phi  $ satisfies
	\begin{equation}\notag
		{\mathscr{W}_\xi }\left( {\xi ,0} \right) = 0,\quad {\mathscr{W}_\kappa }\left( {\xi ,0} \right) = \omega .
	\end{equation}
	Moreover, the transformation $ \phi $ and the Hamiltonian function $ \mathscr{W} $ satisfy the following  estimates:
	\[\left\|\phi \left( \zeta  \right) - {\rm id} \right\|_{{\frac{\sigma}{4} ,\frac{\sigma}{4} }} \leqslant \frac{2{\sigma ^{\frac{2}{\eta} }}}{K},\quad \left\|{\phi _\zeta }\left( \zeta  \right) - {\rm Id}\right\|_{{\frac{\sigma}{4} ,\frac{\sigma}{4} }} \leqslant  \frac{8{\sigma ^{\frac{2}{\eta}-1 }}}{{ K}},\]
	and
	\[\left\|{\mathscr{W}_{\kappa \kappa }}\left( \zeta \right) - \mathscr{Q}\left( \zeta \right)\right\|_{{\frac{\sigma}{4} ,\frac{\sigma}{4} }} \leqslant \frac{2{\sigma ^{\frac{2}{\eta} -1}}}{{ K}}.\]
\end{theorem}
\begin{remark}
The full-dimensional KAM torus obtained in Theorem \ref{FULLT1} is almost periodic in time and preserves the prescribed frequency $\omega\in\mathbb{R}^{\mathbb{Z}}$ of the unperturbed Hamiltonian system. This differs significantly from the result in \cite[Theorem 3.1]{arXiv:2306.08211}, where the authors investigated the linearization of $C^\infty$ perturbed vector fields over $\mathbb{T}^\infty$ with frequency drifting, assuming the same spatial structure. Moreover, the frequency $ \omega\in \mathbb{R}^{\mathbb{Z}} $ here does not require any spectral asymptotics. For example, this is in contrast to the works of Bourgain  \cite{MR2180074}  and Kuksin--P\"oschel  \cite{MR1370761}. Additionally, the components of the frequency can also tend to infinity.
\end{remark}

\begin{remark}\label{FULLRe1,1}
	  The  unperturbed analytic system might be non-integrable, which is an important distinction from the majority of known results. To be more precise, the coefficient of the quadratic term in $ \mathscr{H}(x,y) $ with respect to $ y $ may depend on $ x $, e.g., 
\[\mathscr{H}(x,y) = \underbrace {\left\langle {\omega ,y} \right\rangle  + \left\langle { \mathscr{C}\left( x \right)y },y \right\rangle }_{\text{the non-integrable unperturbed system}} + \underbrace  {\text{\rm higher-order terms}}_{\text{the perturbation}}\]
for $ x,y $ on suitable weighted spaces.
To ensure  nondegeneracy, a simple and applicable case is the infinite-dimensional diagonal constant matrix $ \mathscr{C}(x)=\mathscr{C}={\rm diag}(\mathscr{C}_j)_{j \in \mathbb{Z}} $, where $ 0<\inf_{j \in \mathbb{Z}}|\mathscr{C}_j|\leqslant \sup_{j \in \mathbb{Z}}|\mathscr{C}_j|<+\infty $. 
	The same holds for Theorem \ref{FULLT2}.
\end{remark}

\subsection{KAM with infinite-dimensional weak Diophantine frequency-preserving}
Since the generating function method is essentially Newtonian,  the resulting convergence rate is typically super-exponential. Building on this  key observation, we are able to establish in this section a weaker form of KAM persistence for infinite-dimensional Hamiltonian systems: full-dimensional KAM tori with \textit{weak} Diophantine non-resonance-preserving. We also mention relevant works on weak Diophantine non-resonance conditions, including those by Brjuno \cite{MR377192}, P\"oschel \cite{MR1037110}, and Corsi \textit{et al.} \cite{MR4781767}, as well as the studies by the authors \cite{arXiv:2306.08211,MR4836959}.

As usual, in infinite-dimensional settings, the small divisors, which are the hardest part to deal with in KAM theory, arise from the following analytic homological equation
\begin{equation}\label{FULLtdfc}
	\omega  \cdot {\partial _x}f = g,\quad x \in \mathbb{T}_\sigma ^\infty,
\end{equation}
where $ f\in  {\mathcal{G}_0}\left( {\mathbb{T}_\sigma ^\infty } \right) $, $ g\in  {\mathcal{G}_0}\left( {\mathbb{T}_{\sigma+\rho} ^\infty } \right) $ with $ \sigma,\rho>0 $, and $ \omega \in \mathbb{R}^{\mathbb{Z}} $ is a fixed non-resonant frequency. To characterize the effect of the non-resonance condition, we introduce the following \textit{control function}.

\begin{definition}[Control function]\label{FULLCONTROL}
	A function $\mathscr{E}:\mathbb{R}^+\to\mathbb{R}^+$ is called a control function if it is monotonically decreasing, continuous, and there exists a positive sequence $\{\delta_m\}_{m\in\mathbb{N}}$ such that
	\begin{equation}\label{FULLweakcon}
		\sum\limits_{m = 0}^\infty  {{\delta _m}}  <  + \infty ,\quad \sum\limits_{m = 0}^\infty  {{{\mathscr{E}}^{ - 1}}\big( {{{\rm e}^{{2^m}{\delta _m}}}} \big)}   <  + \infty.
	\end{equation}
	
\end{definition}

We will explain two boundedness conditions for the control function in the subsequent Comment \ref{com:c1} below Theorem \ref{FULLT2}. As we will see later, using the control function to directly solve the homological equation under a general non-resonance condition beyond the Diophantine type greatly simplifies the analysis of the KAM iteration.

\begin{definition}[Infinite-dimensional weak Diophantine frequency]\label{weakdio}
	A frequency $ \omega \in {\mathbb{R}}^\mathbb{Z} $ is said to satisfy the weak Diophantine condition 	if the unique solution $f\in\mathcal{G}_0(\mathbb{T}_\sigma^\infty)$ of \eqref{FULLtdfc} satisfies the estimate
	\[{\left\| f \right\|_\sigma } \leqslant {\mathscr{E}}\left( \rho  \right){\left\| g \right\|_{\sigma  + \rho }},\]
	where $ {\mathscr{E}} $ is a control function as defined in Definition \ref{FULLCONTROL}, independent of $ f$ and $g $.
\end{definition}

With the above notions, our second main result reads:

\begin{theorem}[KAM via the infinite-dimensional weak  Diophantine non-resonance condition]\label{FULLT2}
Let $ \sigma >0 $ be sufficiently large\footnote{For example, we may choose $ \sigma  \geqslant 64\sum\nolimits_{m = 0}^\infty  {{{\mathscr{E}}^{ - 1}}\left( {{{\rm e}^{{2^m}{\delta _m}}}} \right)} $, where $ {\mathscr{E}} $ and $ \{\delta_m \}_{m \in \mathbb{N}} $ are defined in Definition \ref{weakdio}  for the weak Diophantine frequency $ \omega $;  see Section \ref{FULLSEC4} for details.}, and assume that the  frequency $ \omega  \in {{\mathbb{R}}^\mathbb{Z}} $ satisfies the infinite-dimensional weak Diophantine condition in Definition \ref{weakdio}. Then there exists $ \epsilon^*>0 $ such that the following statements  hold for every $ 0<\epsilon<\epsilon^* $.	Assume that $ \mathscr{H}(x,y) $ is a real analytic Hamiltonian function defined on $ {\mathscr{D}_{\sigma ,\sigma }} $, which is $ 1 $-periodic in the variables $ (x_j)_{j \in \mathbb{Z}} $, and satisfies
	\begin{equation}\notag 
		\left\|\mathscr{H}\left( {x,0} \right) - \int_{{\mathbb{T}^\infty }} {\mathscr{H}\left( {\xi ,0} \right){\rm d}\xi } \right\|_{\sigma },\;	\left\|{\mathscr{H}_y}\left( {x,0} \right) - \omega \right\|_{\sigma },\;	\left\|{\mathscr{H}_{yy}}\left( {x,y} \right) - \mathscr{Q}\left( {x,y} \right)\right\|_{{\sigma ,\sigma }} \leqslant \epsilon,
	\end{equation}
	where $ \mathscr{Q}\left( {x,y} \right) \in {\mathbb{C}^{\mathbb{Z} \times \mathbb{Z}}} $ is a Hermitian  and analytic matrix-valued function on $ {\mathscr{D}_{\sigma ,\sigma }} $ satisfying
	\[\left\|\mathscr{Q}\left( {x,y} \right)\right\|_{{\sigma ,\sigma }} \leqslant M,
	\quad \left\|{\left( {\int_{{\mathbb{T}^\infty }} {\mathscr{Q}\left( {x,0} \right){\rm d}x} } \right)^{ - 1}}\right\|_{{\mathscr{I}_\sigma^\infty}} \leqslant M\]
	with some $ M>0 $. 
	
	Under these assumptions,  there exists a real analytic symplectic transformation $ z = \phi \left( \zeta  \right) $ of the form
	\begin{equation}\notag
		z = \left( {x,y} \right),\quad \zeta  = \left( {\xi ,\kappa } \right),\quad x = u\left( \xi  \right),\quad y = v\left( \xi  \right) + \left\langle u_\xi {\left( \xi  \right)^{ - 1}},\kappa \right\rangle ,
	\end{equation}
	mapping $ {\mathscr{D}_{\sigma /4,\sigma /4}} $ into $ {\mathscr{D}_{\sigma ,\sigma }} $, such that $ u\left( \xi  \right) - {\rm id}  $ and $ v\left( \xi  \right) $ are $1$-periodic in all variables, and the Hamiltonian function $ \mathscr{W} = \mathscr{H} \circ \phi  $ satisfies
	\begin{equation}\notag
		{\mathscr{W}_\xi }\left( {\xi ,0} \right) = 0,\quad {\mathscr{W}_\kappa }\left( {\xi ,0} \right) = \omega .
	\end{equation}
	Moreover, the transformation  $ \phi $ and the Hamiltonian function $ \mathscr{W} $ satisfy the estimates
	\begin{equation}\notag
		\left\|\phi \left( \zeta  \right) - {\rm id} \right\|_{{\frac{\sigma}{4},\frac{\sigma}{4}}} ,\left\|{\phi _\zeta }\left( \zeta  \right) - {\rm Id}\right\|_{{\frac{\sigma}{4},\frac{\sigma}{4}}} ,\left\|{\mathscr{W}_{\kappa \kappa }}\left( \zeta \right) - \mathscr{Q}\left( \zeta \right)\right\|_{{\frac{\sigma}{4},\frac{\sigma}{4}}} \leqslant \epsilon.
	\end{equation}
\end{theorem}

Let us  make some further  comments on Theorem \ref{FULLT2}.

\begin{enumerate}[label=(C\arabic*), ref=(C\arabic*)]
	\item \label{com:c1}  Two boundedness conditions in  \eqref{FULLweakcon}   ensure the existence of the contraction sequence $ \left\{ {{\widetilde{\sigma} _\nu }} \right\}_{\nu  \in \mathbb{N}}  $ in the KAM iteration and the uniform convergence of the KAM error $ \left\{ {{\widetilde{\varepsilon} _\nu }} \right\}_{\nu  \in \mathbb{N}}  $, respectively, see Section \ref{FULLSEC4} for definitions. However, in non-analytic cases (such as Gevrey regularity or lower $C^\infty$ regularity; note that at least $C^\infty$ regularity is required for the infinite-dimensional case, due to counterexamples by Herman \cite{MR0874026}, Cheng--Wang \cite{MR3061774}, and Wang \cite{Wangarxiv}), these boundedness conditions need to change accordingly. We also refer to a completely different technique in \cite{arXiv:2306.08211},  where the authors obtained  equilibrium conditions regarding  regularity and the non-resonance condition without the action variable $ y $, in the sense of preserving full-dimensional invariant tori with \textit{frequency-drifting}.

	\item      By utilizing the weak Diophantine condition, we can still  achieve the frequency-preserving KAM persistence in Theorem \ref{FULLT2}. However, explicitly characterizing the smallness conditions (e.g., $ \epsilon $ in terms of $   \sigma$) may be challenging. This is in contrast to the quantitative estimates provided in Theorem \ref{FULLT1}.
	
		\item  We show that the weak Diophantine condition in Definition \ref{weakdio} is indeed \textit{weaker} than the classical one. Consider the Diophantine case (Definition \ref{FULLDio}) for the homological equation \eqref{FULLtdfc}. With Lemma \ref{FULLtdfcyl}, we obtain 
		\begin{align*}
			\exp \left( {\frac{\tau }{{{\rho ^{\frac{1}{\eta} }}}}\log \left( {\frac{\tau }{\rho }} \right)} \right)   \leqslant \exp \left( {\frac{{\widetilde \tau }}{\rho }} \right) - 1: = {\mathscr{E}}\left( \rho  \right),
		\end{align*}
		provided some $ \widetilde{\tau}>0 $ independent of $ \rho $. Consequently, we have $ {{\mathscr{E}}^{ - 1}}\left( \rho  \right) = {\widetilde \tau }({\log \left( {1 + \rho } \right)})^{-1} $. By choosing $ {\delta _m}: = {2^{ - m}}{m^2} $ with $ m \in \mathbb{N} $, we can verify that $  \sum\nolimits_{m = 0}^\infty  {{\delta _m}}  <  + \infty   $, and
		\[\sum\limits_{m = 0}^\infty  {{{\mathscr{E}}^{ - 1}}\big( {{{\rm e}^{{2^m}{\delta _m}}}} \big)}  = \sum\limits_{m = 0}^\infty  {{{\mathscr{E}}^{ - 1}}\big( {{{\rm e}^{{m^2}}}} \big)}  = \sum\limits_{m = 0}^\infty  {\frac{{\widetilde \tau }}{{\log \left( {1 + {{\rm e}^{{m^2}}}} \right)}}}  \leqslant \widetilde \tau \left( {\frac{1}{{\log 2}} + \sum\limits_{m = 1}^\infty  {\frac{1}{{{m^2}}}} } \right) <  + \infty .\]
		Thus, both conditions in \eqref{FULLweakcon} are satisfied. Therefore, our weak Diophantine non-resonance condition  \textit{covers} the classical Diophantine case.  
		
		\item \label{com:c4}   We further investigate the relationship between arithmetical properties of frequencies and homological equations.	Assume that the non-resonant frequency $ \omega \in \mathbb{R}^\mathbb{Z} $ satisfies 
		\begin{equation}\label{FULLdididi}
			\left| {\left\langle {k,\omega } \right\rangle } \right| > \frac{\gamma }{{{\mathscr{R}} \big( {{{\left| k \right|}_\eta }} \big)}},\quad \gamma>0, \quad \forall 0 \ne k \in \mathbb{Z}_ * ^\infty ,
		\end{equation}
		where $ {\mathscr{R}} :\left[ {1, + \infty } \right) \to {\mathbb{R}^ + } $ is an approximation function, i.e., it is continuous, strictly monotonically  increasing, and tends to positive infinity. We further assume that
		\begin{equation}\label{FULLEEEE}
			\mathop {\sup }\limits_{x \geqslant 1} \left\{{\mathscr{R}} \left( x \right){{\rm e}^{ - \rho x}}\right\} \leqslant {\mathscr{E}}\left( \rho  \right)
		\end{equation}
		with properties of the control function $ {\mathscr{E}}(\rho) $ given in \eqref{FULLweakcon}. Then the frequency $ \omega \in \mathbb{R}^\mathbb{Z} $ above satisfies the weak Diophantine condition, due to the estimate for the unique solution $ f\in  {\mathcal{G}_0}\left( {\mathbb{T}_\sigma ^\infty } \right) $ of the  homological equation \eqref{FULLtdfc}:
		\begin{align}
			{\left\| f \right\|_\sigma } &= \sum\limits_{0 \ne k \in \mathbb{Z}_ * ^\infty } {| {{{\widehat f}_k}} |{{\rm e}^{\sigma {{\left| k \right|}_\eta }}}}  = \sum\limits_{0 \ne k \in \mathbb{Z}_ * ^\infty } {\frac{{\left| {{{\widehat g}_k}} \right|}}{{\left| {\left\langle {k,\omega } \right\rangle } \right|}}{{\rm e}^{\sigma {{\left| k \right|}_\eta }}}}  \leqslant {\gamma ^{ - 1}}\sum\limits_{0 \ne k \in \mathbb{Z}_ * ^\infty } {\left| {{{\widehat g}_k}} \right|{\mathscr{R}} \big( {{{\left| k \right|}_\eta }} \big){{\rm e}^{\sigma {{\left| k \right|}_\eta }}}} \notag \\
			& \leqslant {\gamma ^{ - 1}}{\mathscr{E}}\left( \rho  \right)\sum\limits_{0 \ne k \in \mathbb{Z}_ * ^\infty } {\left| {{{\widehat g}_k}} \right|{{\rm e}^{\left( {\sigma  + \rho } \right){{\left| k \right|}_\eta }}}}  = {\gamma ^{ - 1}}{\mathscr{E}}\left( \rho  \right){\left\| g \right\|_{\sigma  + \rho }}.\notag
		\end{align}
		Below we give an almost \textit{critical} case. For $ \lambda>0 $, consider the approximation function
		\begin{equation}\label{FULLDeltalog}
			{\mathscr{R}} \left( x \right) \sim \exp \left( {x{{\left( {\log x} \right)}^{ - 1 - \lambda }}} \right),\quad x \to  + \infty .
		\end{equation}
		Then we can take  the control function as
		\[{\mathscr{E}}\left( \rho  \right) = \exp \left( {\exp \left( {{\rho^{ - \widetilde \lambda }}} \right)} \right)\]
		for some $ \widetilde \lambda  \in \left( {0,1} \right) $, see    Lemma \ref{lambdapiao} in the Appendix for details. Let $ {\delta _m} = {\left( {m + 1} \right)^{ - 2}} $ for $ m \in \mathbb{N} $. Then it follows that $ \sum\nolimits_{m = 0}^\infty  {{\delta _m}}  <  + \infty  $, and
		\[\sum\limits_{m = 0}^\infty  {{{\mathscr{E}}^{ - 1}}\big( {{{\rm e}^{{2^m}{\delta _m}}}} \big)}  \leqslant C\sum\limits_m {\frac{1}{{{{\left( {\log \log \left( {{{\rm e}^{{2^m}{\delta _m}}}} \right)} \right)}^{\frac{1}{\widetilde \lambda}  }}}}}  \leqslant C\sum\limits_m {\frac{1}{{{m^{\frac{1}{\widetilde \lambda} }}}}}  <  + \infty \]
		due to $ \widetilde \lambda^{-1} \in (1,+\infty) $. This implies that all frequencies $ \omega $ satisfying  \eqref{FULLdididi} along with \eqref{FULLDeltalog} are  of the weak Diophantine type. 
		
	\item \label{com:c5}  In fact, the approximation function \eqref{FULLDeltalog} mentioned in Comment \ref{com:c4} can be further weakened to
	\begin{equation}\label{FULLloooog}
		{\mathscr{R}} \left( x \right) \sim \exp \left( { {x}{\left( {\log x} \right)^{-1} \cdots {{( {\underbrace {\log  \cdots \log }_\ell x} )}^{-1 - \lambda }}}} \right),\quad x \to  + \infty
	\end{equation}
	where   $ 2 \leqslant \ell  \in {\mathbb{N}^ + } $ and $ \lambda>0 $ are arbitrary. Consequently, the control function $ {\mathscr{E}}(\rho) $ in \eqref{FULLEEEE} can be taken as
	\[{\mathscr{E}}\left( \rho  \right) = \exp \left( {\exp \left( { {\rho^{-1} \left( {\log {\rho ^{ - 1}}} \right)^{-1} \cdots {{( {\underbrace {\log  \cdots \log }_{\ell  - 1}{\rho ^{ - 1}}} )}^{-1 - \bar\lambda }}}} \right)} \right),\quad \rho  \to {0^ + }\]
	for some $ \bar \lambda  \in \left( {0,1} \right) $, in a similar manner. Under this setting, it is verified that the frequencies satisfying \eqref{FULLdididi} along with \eqref{FULLloooog} still exhibit weak Diophantine characteristics.  It is worth noting  that this is almost \textit{optimal} in the finite-dimensional case. Specifically, the  parameter $ \lambda >0$ in \eqref{FULLloooog} cannot degenerate to $ 0 $, as this would cause the optimal (at least for the $ 2 $-dimensional case) Brjuno condition \cite{MR377192} for KAM persistence to no longer be satisfied:
	\[\int_1^{ + \infty } {\frac{{\log \mathscr{R}\left( x \right)}}{{{x^2}}}{\rm{d}}x}  \geqslant C\int_M^{ + \infty } {\frac{1}{{x(\log x) \cdots (\log  \cdots \log x)}}{\rm{d}}x}  =  + \infty .\]
	In other words, our weak Diophantine condition is also almost \textit{optimal}.
	
	\item   Our results are of significant physical interest beyond the Aubry--MacKay conjecture discussed in Section \ref{SECMACKAY}. For example, Arnaiz \cite{MR4105369} directly applied KAM theorems to study semiclassical KAM and renormalization theorems based on counterterms (or modifying terms). This enabled the characterization of certain semiclassical measures and quantum limits. Consequently, our infinite-dimensional KAM Theorems \ref{FULLT1} and \ref{FULLT2} will play an important role in further addressing such physically related problems, \textit{both in the almost periodic sense and in the frequency-preserving sense.}
\end{enumerate}

\subsection{Further discussion on the  spatial structure}\label{SEC24}
As previously mentioned, certain  spatial structures are essential  in the infinite-dimensional Hamiltonian context.  This has given rise to   numerous  challenges  and  many open questions. For example, consider the reducibility of a linear Schr\"odinger equation subject to a small, unbounded, almost periodic perturbation on the thickened torus $\mathbb{T}_{\sigma}^\infty$ (see \eqref{FUUHUAN}) with analyticity, as discussed in \cite{MR4201442}. It is unknown whether the analyticity radius there can be weakened. For instance, could one take $|\operatorname{Im} x_j| \leqslant \sigma \log(1+\langle j\rangle)^p$ for some $p\gg 1$ instead of $|\operatorname{Im} x_j| \leqslant \sigma \langle j\rangle^\eta$, for all $ j \in \mathbb{Z} $?

 Furthermore, following  the spirit of Moser, the profound interplay among the non-resonance condition, regularity, and  spatial structures in infinite-dimensional Hamiltonian systems remains largely unexplored.    A recent advancement in this direction can be attributed to \cite{arXiv:2306.08211}, where the authors gave the sharp regularity for  Gevrey and even $ C^\infty $ infinite-dimensional vector fields, ensuring the persistence of full-dimensional tori. Notably, in some cases, the authors observed that the spatial structures are not limited to the usual thickened torus $ \mathbb{T}_{\sigma}^\infty $, provided a specific equilibrium condition is satisfied. We also refer to the recent contribution by Corsi \textit{et al.} \cite{MR4781767}, wherein the profound relationship between the spatial structure and other relevant quantities is explored.     
 
  Here, we briefly demonstrate  that a similar phenomenon holds true for our KAM theorems, thereby addressing the issue in \cite{MR4201442} within the abstract infinite-dimensional Hamiltonian framework (not for the linear Schr\"odinger equation).

Recall the previous notation for the spatial structure in this paper. Let us first modify the thickened torus $ \mathbb{T}_{\sigma}^\infty $ in \eqref{FUUHUAN} to
\[\mathbb{T}_\sigma ^\infty : = \left\{ {x = {{({x_j})}_{j \in \mathbb{Z}}}:\quad {x_j} \in \mathbb{C},\;\operatorname{Re} {x_j} \in \mathbb{T},\;\left| {\operatorname{Im} {x_j}} \right| \leqslant \sigma \varphi\left({{\left\langle j \right\rangle }^\eta }\right), \; j \in \mathbb{Z} } \right\},\]
where $\varphi: \mathbb{R}^+ \to \mathbb{R}^+$ is a monotonically increasing function with $\varphi(1) \geqslant 1$.  Then,  we similarly define the analytic space $ {\mathcal{G}}\left( {\mathbb{T}_\sigma ^\infty } \right)  $ on the torus as in \eqref{FUUJIEXI} with the norm of $ \ell $ given by $ {\left| \ell  \right|_\varphi }: = \sum\nolimits_{j \in \mathbb{Z}} {\varphi \left( {\left\langle j \right\rangle } \right)\left| {{\ell _j}} \right|}  $. It is evident   that $ {\left| \ell  \right|_\varphi } \geqslant \left| \ell  \right|_1: = \sum\nolimits_{j \in \mathbb{Z}} {\left| {{\ell _j}} \right|}  $,  and except for Lemma \ref{FULLtdfcyl} (just for the Diophantine case), the lemmas in Section \ref{FUUSECPRE} still hold.  The only point that needs to be stressed is the definition of the weighted norm of the action variable $ y $, as this is relatively special in the  KAM iteration process, see Section \ref{FULLSEC3} for details. To be more precise, let $ {\left\| y \right\|_\iota^* }: = \sum\nolimits_{j \in \mathbb{Z}} {\left| {{y_j}} \right|{\varphi ^\iota }\left( {\left\langle j \right\rangle } \right)}  $ for some $ \iota>0 $ if $ \varphi $ tends to $ +\infty $, and let $ {\left\| y \right\|^ * }: = {\sup _{j \in \mathbb{Z}}}\left| {{y_j}} \right|  $ if $ \varphi $ is constant (recall that $ \varphi $ is monotonically increasing). The weighted norms for (matrix valued) analytic functions and the infinite-dimensional weak Diophantine non-resonance condition can also be defined in a similar way. With the above replacement, we present the following KAM Theorem \ref{FULLT3} without proof:

\begin{theorem}\label{FULLT3}
Assume that the infinite-dimensional setting is replaced by the one described above, and that the assumptions in Theorem \ref{FULLT2} hold. Then the frequency-preserving full-dimensional KAM torus  survives small perturbations.
\end{theorem}

\section{Application to the Aubry--MacKay conjecture}\label{SECMACKAY}
\subsection{Discrete breathers}\label{SEC31}
Periodic, quasi-periodic, and almost periodic breathers---namely, solutions to nonlinear systems that are periodic, quasi-periodic, or almost periodic in time and localized in space---have long been a subject of significant interest\footnote{Notwithstanding the difference in terminology between KAM theory and physics-related problems, we continue to use the term ``breather'' in this section to correspond with historical conventions in the latter.}. The difficulty of establishing their existence increases progressively from the periodic to the quasi-periodic and, ultimately, the almost periodic case; consequently, the relevant literature becomes correspondingly sparse. As is well known, only a few partial differential equations  admit continuous breathers. In contrast, many partial difference equations, or equivalently, many infinite systems of ordinary differential equations, admit  \textit{discrete} breathers, for which there are some profound reasons related to spectral theory. From the perspectives of physics (e.g., condensed matter physics) and biophysics, 	investigating the latter is both natural and fundamentally important, as such discrete breathers can produce complex  effects of energy focusing and transfer. In this case, the spatial variable is not continuous in space; instead, it varies over some lattice, such as in a crystal.  The foundational literature on the existence of discrete \textit{periodic} breathers is extensive,  culminating in well-established theoretical frameworks. The origins of this line of research can be traced back to the seminal work of Fr\"ohlich \textit{et al.} \cite{MR0833019}. Key contributions in this direction include works by Aubry \cite{MR1264115,MR1353962,MR1464249}, MacKay--Aubry \cite{MR1304442}, Arioli--Gazzola \cite{MR1375653}, 
Arioli \textit{et al.} \cite{Arioli},	Ahn \cite{MR1632649},	 and the references therein, alongside recent advancements such as those by Pelinovsky--Sakovich \cite{MR2997701} and Arioli--Koch \cite{MR4034784}.	In addition, comprehensive surveys on discrete \textit{periodic} breathers can be found in the works of MacKay \cite{MR1353954,MR1808178}, Flach and Willis \cite{MR1607320}, and Aubry \cite{MR2230489}. However, results concerning discrete \textit{quasi-periodic/almost periodic} breathers remain relatively scarce, and the exploration is more challenging. Notable contributions  in this direction  include the comprehensive works of Yuan \cite{MR1889993}, and Fontich \textit{et al.} \cite{MR3353644}, along with the  references cited therein.

\subsection{On the Aubry--MacKay conjecture: Persistence of frequency-preserving almost periodic breathers}\label{SECAMC}
The Aubry--MacKay conjecture was proposed in \cite{MR1304442,MR1353962}. It poses the question of the existence of quasi-periodic breathers, for instance, in the simplest case of a $ 1 $-dimensional chain of identical one degree of freedom oscillators with symmetric nearest neighbour coupling \cite{MR1304442}:
\[\frac{{{{\rm d}^2}{x_n}}}{{{\rm d}{t^2}}} + V' \left( {{x_n}} \right) = \alpha \left( {{x_{n + 1}} - 2{x_n} + {x_{n - 1}}} \right),\quad n \in \mathbb{Z},\]
where the local potential $ V $ satisfies $ V'\left( 0 \right) = 0 $ and $ V''\left( 0 \right) > 0 $, and $ \alpha>0 $ is a sufficiently small constant. This particularly simple model is typically referred to as the \textit{Klein-Gordon lattice} (see	\cite{MR1808178} for example). As mentioned in Section \ref{SEC31}, the existence of quasi-periodic and almost periodic breathers has been resolved for some simple systems using  KAM techniques, as conceived in \cite{MR1304442,MR1353962}. However, the frequency-preserving case has not yet been addressed, which is a crucial issue from the perspective of the persistence of dynamical systems. Moreover, for the almost periodic case, it is preferable to utilize the infinite-dimensional Diophantine non-resonance condition proposed by Bourgain (see Definition \ref{FULLDio}), as explained by Corsi \textit{et al.} \cite{MR4781767}. Therefore, many fundamental questions remain to be explored regarding the Aubry--MacKay conjecture \cite{MR1304442,MR1353962}.

In this paper, using Bourgain's non-resonance condition, we first investigate this conjecture for frequency-preserving almost periodic breathers and, more generally, for networks of weakly coupled oscillators:
\begin{equation}\label{MACXT}
	\frac{{{{\rm d}^2}{x_n}}}{{{\rm d}{t^2}}} + V'\left( {{x_n}} \right) = \varepsilon_n  {W'\left( {{x_{n + 1}} - {x_n}} \right) - \varepsilon_{n-1}W'\left( {{x_n} - {x_{n - 1}}} \right)} ,\quad n \in \mathbb{Z},
\end{equation}
given  a local potential $ V $ and a  coupling potential $ W $. To provide the  reader with an intuitive grasp of the coupling structure, the interaction effects are illustrated in Figure \ref{FIG1}.
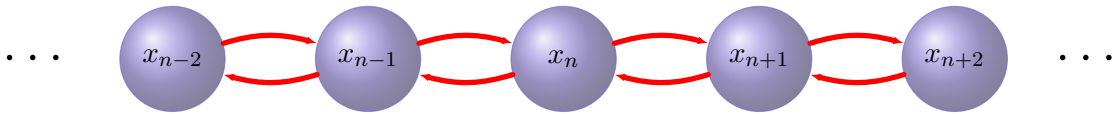
\begin{figure}[H]
	\centering
	\begin{tikzpicture}[
		arr/.style={-{Latex[length=4pt,bend]}, line cap=round}
		]
		
		\pgfmathsetmacro{\dx}{2.6}   
		\pgfmathsetmacro{\wNear}{2.0pt}
		\pgfmathsetmacro{\wMid}{1.3pt}
		\pgfmathsetmacro{\wFar}{0.7pt}
		\pgfmathsetmacro{\bend}{6mm}
		
		\pgfmathsetmacro{\leftB}{-1.3}
		\pgfmathsetmacro{\rightB}{4*\dx+1.3}
		\pgfmathsetmacro{\shift}{-(\leftB+\rightB)/2}  
		\begin{scope}[shift={(\shift,0)}]
			
			\foreach \idx [count=\i from 0] in {n-2,n-1,n,n+1,n+2}{
				\begin{scope}[opacity=0.58]
					\node[ball3Donly] (ball\i) at (\i*\dx,0) {};
				\end{scope}
					\node[text=black] at (\i*\dx,0) {$\boldsymbol{x_{\idx}}$};
			}
			
			\node[scale=1.8, black] at (-1.8,0) {$\cdots$};
			\node[scale=1.8, black] at (4*\dx+1.8,0) {$\cdots$};
			
			\draw[arr,red,line width=\wNear] (ball0) to[bend left=\bend] (ball1);
			\draw[arr,red,line width=\wNear] (ball1) to[bend left=\bend] (ball0);
			\draw[arr,red,line width=\wNear] (ball1) to[bend left=\bend] (ball2);
			\draw[arr,red,line width=\wNear] (ball2) to[bend left=\bend] (ball1);
			\draw[arr,red,line width=\wNear] (ball2) to[bend left=\bend] (ball3);
			\draw[arr,red,line width=\wNear] (ball3) to[bend left=\bend] (ball2);
			\draw[arr,red,line width=\wNear] (ball3) to[bend left=\bend] (ball4);
			\draw[arr,red,line width=\wNear] (ball4) to[bend left=\bend] (ball3);

		\end{scope}
	\end{tikzpicture}
	\caption{Weak coupling in system \eqref{MACXT}: only neighboring oscillators are coupled.}
	\label{FIG1}
\end{figure}

In this context, we impose the following basic assumptions:

\begin{enumerate}[label=(A\arabic*), ref=(A\arabic*)]

\item \label{con:a1} The local potential $ V(z) $ and the coupling potential $ W(z) $ are analytic in the complex strip $ \{z \in \mathbb{C}:\left|\operatorname{Im}{z}\right|\leqslant\delta_0 \} $ for some $ \delta_0>0 $, and $ V(x) $ is positive definite for  real $ x $.

\item \label{con:a2} $ V(0)=0 $, and $ W(x)=\mathcal{O}\left(|x|^3\right) $ as  $ |x|\to 0^+ $.

\item \label{con:a3} For a given compact interval $ \mathcal{I}\subset \mathbb{R}^+ $ and any $ h \in \mathcal{I} $, the equation $ {y^2}/2 + V\left( x \right) = h $ defines a simple closed curve $ \Gamma(h) $ which encloses the origin $ (0,0) $ in the $ (x,y) $-plane.

\item \label{con:a4}  Let $ \rho=\rho(h) $ be the area of the region enclosed by the closed curve $ \Gamma(h) $, i.e., 
\[\rho \left( h \right) = \oint_{\frac{1}{2} y^2 + V\left( x \right) = h} {y{\rm d}x} .\]
We assume that under the potential $ V $, for any $ h \in \mathcal{I} $,  $ \rho '\left( h \right) \ne 0 $ and $\rho ''\left( h \right) \ne 0 $, and $ \rho^{-1}(h) $ is analytic.

\end{enumerate}

For the sake of simplicity, we adopt the notation that $ a \prec b $ with $ a=\left(a_n\right)_{n \in \mathbb{Z}} $ and $ b=\left(b_n\right)_{n \in \mathbb{Z}} $ implies $ |a_n|<|b_n |$ for all $ n \in \mathbb{Z} $. A similar principle applies to infinite-dimensional matrices. Now, our main result in this section is stated as follows:

\begin{theorem}\label{MACDINGLI}
	Assuming \ref{con:a1}--\ref{con:a4}, there exists some $ \varepsilon^*=\left(\varepsilon_n^*\right)_{n \in \mathbb{Z}} $ with $ \|\varepsilon^*\|_{1}:=\sum\nolimits_{n \in \mathbb{Z}} {\left| {{\varepsilon^* _n}} \right|} >0 $  sufficiently small, such that for all $ 0\prec\varepsilon=\left(\varepsilon_n\right)_{n \in \mathbb{Z}}\prec\varepsilon^* $,  system \eqref{MACXT} admits a large family of frequency-preserving almost periodic breathers, where the frequencies have large magnitudes. Moreover, these frequencies are of Bourgain's Diophantine type as defined in Definition \ref{FULLDio}. If \ref{con:a4} holds for all $ h \in \overline {\mathcal{I}} $, then such frequency-preserving almost periodic breathers can have large, small, and mixed frequency-magnitudes.
\end{theorem}
\begin{remark}
	As remarked by Yuan \cite[Remark 1.2]{MR1889993}, one may not expect the perturbed breathers to preserve their original frequency, in light of Bourgain's work \cite{MR1470416}. However, the nondegeneracy here ensures precisely this invariance. Indeed, one could also anticipate that the quasi-periodic breathers in \cite{MR1889993} would preserve their prescribed frequencies.
\end{remark}

\begin{proof}
Let $x_n'=y_n$ for $n\in\mathbb{Z}$. System \eqref{MACXT} can then be written as a Hamiltonian system with the Hamiltonian
	\begin{equation}\label{MACHAM}
		H = \sum\limits_{n \in \mathbb{Z}} {\left( {\frac{1}{2}y_n^2 + V\left( {{x_n}} \right) + \varepsilon_n W\left( {{x_{n + 1}} - {x_n}} \right)} \right)} .
	\end{equation}
This identification is justified since, for all $ n \in \mathbb{Z} $, we have
	\[\left\{ \begin{gathered}
		{x_n'} = {\left( {{H_y}} \right)_n} = {y_n}, \hfill \\
		{y_n'} = {\left( { - {H_x}} \right)_n} =  - V'\left( {{x_n}} \right) + {\varepsilon _n}W'\left( {{x_{n + 1}} - {x_n}} \right) - {\varepsilon _{n - 1}}W'\left( {{x_n} - {x_{n - 1}}} \right) = {x_n''}. \hfill \\ 
	\end{gathered}  \right.\]
Next,	we  perform the standard reduction to action-angular variables following Arnold \cite{MR690288}. In view of Assumption \ref{con:a4}, we can implicitly define a locally analytic function $ H_0(\rho) $ by 
	\begin{equation}\label{MACrho}
		\rho  = \oint_{\frac{1}{2}{y^2} + V\left( x \right) = {H_0}\left( \rho  \right)} {y{\rm d}x}, 
	\end{equation}
	i.e., $ \rho(h)  = H_0^{ - 1}(h) $. Now, define a generating function $ S(x,\rho) $ as
	\[S\left( {x,\rho } \right) = \int_{{\Gamma ^ * }} {y{\rm d}x} ,\]
	provided that $ \Gamma ^* $ is the part of the closed curve $ {{y^2}/2 + V\left( x \right) = {H_0}\left( \rho  \right)} $ connecting the $ y $-axis with point $ (x,y) $, oriented clockwise. By constructing the locally analytic mapping $ {\psi _0}:\left( {\theta ,\rho } \right) \mapsto \left( {x,y} \right) $ with
	\[{S_x}\left( {x,\rho } \right) = y,\quad {S_\rho }\left( {x,\rho } \right) = \theta ,\]
	we can verify that
	\[{\rm d}x \wedge {\rm d}y = {\rm d}x \wedge \left( {{S_{xx}}{\rm d}x + {S_{x\rho }}{\rm d}\rho } \right) = {S_{x\rho }}{\rm d}x \wedge {\rm d}\rho  = \left( {{S_{\rho x}}{\rm d}x + {S_{\rho \rho }}{\rm d}\rho } \right) \wedge {\rm d}\rho  = {\rm d}\theta  \wedge {\rm d}\rho ,\]
	which implies that the mapping 
	\[\Psi :\quad {\psi _0}\left( {{\theta _n},{\rho _n}} \right) = \left( {{x_n},{y_n}} \right),\quad n \in \mathbb{Z}\]
	is indeed a symplectic one. Then, the Hamiltonian in \eqref{MACHAM} can be transformed into
	\begin{equation}\label{MACHAM2}
		H = H\left( {\theta ,\rho } \right) = \sum\limits_{n \in \mathbb{Z}} {{H_0}\left( {{\rho _n}} \right)}  +  \sum\limits_{n \in \mathbb{Z}} \varepsilon_n{W\left( {{x_{n + 1}}(\theta_{n+1}, \rho_{n+1}) - {x_n}(\theta_n, \rho_n)} \right)} .
	\end{equation}
	For any given $ \xi  = {\left( {{\xi _n}} \right)_{n \in \mathbb{Z}}} \in {\left( {H_0^{ - 1}\left( \mathcal{I} \right)} \right)^o} $, we let $ \rho  = {\left( {{\rho _n}} \right)_{n \in \mathbb{Z}}} = {\left( {{I_n} + {\xi _n}} \right)_{n \in \mathbb{Z}}} = I + \xi  $, and expand $ {H_0}\left( {{\rho _n}} \right) $ into its Taylor expansion with respect to $ I_n $ as 
	\begin{equation}\label{MACH2}
		{H_0}\left( {{\rho _n}} \right) = {H_0}\left( {{I_n} + {\xi _n}} \right) = {H_0}\left( {{\xi _n}} \right) + {{H}_0'}\left( {{\xi _n}} \right){I_n} + \frac{1}{2}{{H}_0''}\left( {{\xi _n}} \right)I_n^2 + \mathcal{O}\left( {{{\left| {{I_n}} \right|}^3}} \right),\quad n \in \mathbb{Z},
	\end{equation}
	where $ \left|I_n\right| \leqslant r \xi_n $  for some $ 0<r<1/2 $. Then, the Hamiltonian in \eqref{MACHAM2}  becomes 
	\begin{equation}\label{MACHA3}
		H = H\left( {\theta ,I,\xi } \right) =  \left\langle {\omega \left( \xi  \right),I} \right\rangle  + \frac{1}{2}\left\langle {A\left( \xi  \right)I,I} \right\rangle+ \mathcal{O}\left( {{{\left| I \right|}^3}} \right) + \widetilde{P}(\theta,I,\xi),
	\end{equation}
	where
	\[\widetilde{P}(\theta,I,\xi)=\sum\limits_{n \in \mathbb{Z}} \varepsilon_n{W\left( {{x_{n + 1}}(\theta_{n+1}, I_{n+1},\xi_{n+1}) - {x_n}(\theta_{n}, I_{n},\xi_{n})} \right)}.\]
	Here, $ \omega \left( \xi  \right) = {\left( {{{H}_0'}\left( {{\xi _n}} \right)} \right)_{n \in \mathbb{Z}}} $ is a frequency mapping, and 
	\[A\left( \xi  \right) = \left( {\begin{array}{*{20}{c}}
			\ddots &{}&{} \\ 
			{}&{{{H}_0''}\left( {{\xi _n}} \right)}&{} \\ 
			{}&{}& \ddots  
	\end{array}} \right)\]
	is an infinite-dimensional diagonal matrix. Moreover, if there exists a universal constant $ c>1 $ such that $ {c^{ - 1}} \leqslant \left| {{{H}_0''}\left( {{\xi _n}} \right)} \right| \leqslant c $ for all $ n \in \mathbb{Z} $, then $ A(\xi) $ is invertible. Indeed, this follows from Assumption \ref{con:a4} and the Lagrange inversion theorem (see, e.g., \cite{MR3534068,MR4669236}) applied on $\mathcal{I}$.   As Yuan \cite{MR1889993} has pointed out, for infinite-dimensional Hamiltonian systems reduced using action-angular variables, the application of infinite-dimensional KAM theorems (e.g., see P\"oschel \cite{MR1037110}) that directly yield full-dimensional tori can establish the existence of almost periodic breathers. Thus, this conclusion is not surprising. \textit{However, the frequency-preserving case  does not have a direct counterpart.} We will achieve this by utilizing the invertibility of $ A(\xi) $.

  Observe  for the perturbed system \eqref{MACXT}, the perturbation is exclusively reflected in the tail $P$ of the Hamiltonian \eqref{MACHA3}. As for the unperturbed part of \eqref{MACHA3}, it is evident that it possesses a large family of invariant tori, with toral frequencies $\omega(\xi)$ of Bourgain's Diophantine type as $\xi$ varies appropriately. This can be seen either by directly analyzing the first integrals of the unperturbed system in \eqref{MACXT} or by applying KAM theory. Therefore, as long as we can prove that these tori survive perturbations and the frequencies of these tori remain unchanged, we will have obtained a frequency-preserving version of the Aubry--MacKay conjecture. In other words, the solutions starting from these preserved tori are indeed almost periodic breathers of \eqref{MACXT}. In what follows, we shall employ Theorem \ref{FULLT1} to prove this.

	By appropriately choosing $0\prec\varepsilon = (\varepsilon_n)_{n\in\mathbb{Z}}$, i.e., ensuring its components $\varepsilon_n$ decay rapidly enough with respect to $|n|$ and that $\|\varepsilon\|_1 = \sum_{n\in\mathbb{Z}}|\varepsilon_n| \ll 1$, the Hamiltonian in \eqref{MACHA3} becomes analytic in $(\theta,I)$ on a complex domain  
	\[{\mathscr{D}_{\sigma ,r}}: = \left\{ {\left( {\theta ,I} \right):\quad \left| {\operatorname{Im} {\theta _n}} \right| \leqslant \sigma {\left\langle n \right\rangle ^\eta } ,\;\left| {{I_n}} \right| \leqslant r\xi_n,\;n \in \mathbb{Z}} \right\}\]
	with some fixed $ 0<\sigma,r<1/2 $ and $ \eta \geqslant 2 $. For example, $ {\varepsilon _n} \sim q{{\rm e}^{ - R\left| n \right|}} $ with some $ 0<q\ll1 $ and $ R\gg1  $ larger than the analyticity radius of $ W\left( {{x_{n + 1}}({\theta _{n + 1}},{I_{n + 1}}) - {x_n}({\theta _n},{I_n})} \right) $. Then, the perturbation of the Hamiltonian in \eqref{MACHA3} is sufficiently small on the domain $ \mathscr{D}_{\sigma, r} $, in the analytic norm. Now, by fixing large $ \xi $  appropriately, we can obtain  an infinite-dimensional Diophantine frequency $ \omega\in \mathcal{I}^\mathbb{Z} $ with large frequency-magnitude. This is achievable since the set of such frequencies has full measure. Therefore, by applying Theorem \ref{FULLT1} concerning the infinite-dimensional Diophantine non-resonance condition, we prove the existence of a large family of frequency-preserving almost periodic breathers with large frequency magnitudes for \eqref{MACXT}.
	
	In the following, we consider the small-magnitude case and also the mixed-magnitude case. Assume that Assumption \ref{con:a4} holds with all $ h \in \overline{\mathcal{I}} $, i.e., $ H_0(h)=\rho^{-1}(h) $ is analytic on $ \overline{\mathcal{I}} $. This implies that both  $ H_0'(\xi_n) $ and  $ H_0''(\xi_n) $ in \eqref{MACH2} are uniformly bounded whenever $ \xi_n $ is chosen sufficiently close to $ 0 $. Moreover, $ |H_0''(\xi_n)| $ is uniformly bounded from below, as we can establish $ |H_0''(0)|>0 $ by utilizing  $ \rho''(0)\ne0 $ and the Lagrange inversion theorem (which necessitates  $ \rho'(0)\ne0 $). To be more precise, for $ 0<h \ll 1 $, we have 
	\[\rho \left( h \right) = {A_1}h + {A_2}{h^2} +  \cdots  = \sum\limits_{j = 1}^\infty  {{A_j}{h^j}} , \quad A_1,A_2\ne0,\]
	which implies that $ H_0 $ is nondegenerate, i.e.,
	\[{H_0}\left( h \right) = {\rho ^{ - 1}}\left( h \right) = \frac{1}{{{A_1}}}h - \frac{{{A_2}}}{{A_1^3}}{h^2} +  \cdots  = \sum\limits_{j = 1}^\infty  {{B_j}{h^j}} ,\quad B_1,B_2\ne0.\]
	However, in order to guarantee the applicability of our KAM Theorem \ref{FULLT1}, $ \xi_n $ cannot be arbitrarily close to $ 0 $, otherwise the analyticity radius of the action variable $ I_n $ may not be adequate. A suitable choice is $ \xi_n> {\left\langle n \right\rangle ^{ - \varsigma-2  }}$ for all $ n \in \mathbb{Z} $, where $ \varsigma>0 $. In this case,  letting $ \varepsilon $ be sufficiently small in a similar way,  the Hamiltonian in \eqref{MACHA3} becomes analytic in $ (\theta,I) $ on a  complex domain 
	\[{\widetilde{\mathscr{D}}_{\sigma ,r}}: = \left\{ {\left( {\theta ,I} \right):\quad \left| {\operatorname{Im} {\theta _n}} \right| \leqslant \sigma {\left\langle n \right\rangle ^\eta } ,\;\left| {{I_n}} \right| \leqslant {\left\langle n \right\rangle ^{ - \varsigma-2 }},\;n \in \mathbb{Z}} \right\}.\]
	Here, we again note our prior assumption that $ \left| {{I_n}} \right| \leqslant r\xi_n $ for some $ 0<r<1/2 $ and all $ n \in \mathbb{Z} $.
	Note that $ \sum\nolimits_{n \in \mathbb{Z}} {\left| {{I_n}} \right|{{\left\langle n \right\rangle }^\varsigma }}  \leqslant \sum\nolimits_{n \in \mathbb{Z}} {{{\left\langle n \right\rangle }^{ - 2}}}  <  + \infty  $. Consequently,  our KAM Theorem \ref{FULLT1}  guarantees  the existence of a large family of frequency-preserving almost periodic breathers with small frequency-magnitudes for  \eqref{MACXT}. Finally, the mixed-magnitude type can be achieved by simultaneously incorporating the cases of large and small magnitudes.
	
	This completes the proof of Theorem \ref{MACDINGLI}.
\end{proof}
In fact, from the standpoint of KAM techniques, the coupled oscillator setting we consider can be extended  to the lattice $\mathbb{Z}^{d}$ for any $d>1$ without substantial additional difficulty.  
To keep the exposition focused on the original physical model investigated by Aubry--MacKay, however, we have chosen to restrict our analysis to this case.  
Interested reader can establish the existence of frequency-preserving breathers on $\mathbb{Z}^{d}$ in an entirely analogous manner.

\subsection{Broader applications: More general networks of  coupled oscillators} 
One can observe from the proof of Theorem \ref{MACDINGLI} that the specific form of the coupling potential $W$ is not essential, thanks to the effectiveness of our frequency-preserving KAM Theorem \ref{FULLT1} in the infinite-dimensional setting. Consequently, by employing a similar methodology, we can derive a more comprehensive version of Theorem \ref{MACDINGLI}, as presented below. For the sake of brevity, we omit the proof. \textit{This allows for more intricate connections between oscillators, which are not limited to just two adjacent ones.}

\begin{theorem}\label{TH22}
	Consider the following networks of coupled oscillators
	\begin{equation}\label{MACXIT2}
		\frac{{{{\rm d}^2}{x_n}}}{{{\rm d}{t^2}}} + V'\left( x_n \right) = \sum\limits_{p =  - {m_{1,n}}}^{ - 1} {{\varepsilon _{n,p}}{W_{n,p}^\prime} \left( {{x_{n + p}} - {x_n}} \right)}  - \sum\limits_{p=1}^{{m_{2,n}}} {{\varepsilon _{n,p}}{W_{n,p}^\prime} \left( {{x_{n + p}} - {x_n}} \right)} ,\quad n \in \mathbb{Z},
	\end{equation}
	where $ {\left\{ {{m_{1,n}}} \right\}_{n \in \mathbb{Z}}} $ and $ {\left\{ {{m_{2,n}}} \right\}_{n \in \mathbb{Z}}} $ are positive integer sequences.
	Assuming \ref{con:a1}--\ref{con:a4}, there exists some $ \varepsilon^*=\left(\varepsilon_{n,p}^*\right)_{n,p \in \mathbb{Z}}$ with $ \|\varepsilon^*\|_{ 1}:=\sum\nolimits_{n,p \in \mathbb{Z}} {\left| {{\varepsilon^* _{n,p}}} \right|} >0 $   sufficiently small, such that for all $ 0 \prec \varepsilon=\left(\varepsilon_{n,p}\right)_{n,p \in \mathbb{Z}}\prec\varepsilon^* $,  system \eqref{MACXIT2} admits a large family of frequency-preserving almost periodic breathers with large frequency-magnitudes. Moreover, these frequencies are of Bourgain's Diophantine type as defined in Definition \ref{FULLDio}. If \ref{con:a4} holds for all $ h \in \overline {\mathcal{I}} $, then such frequency-preserving almost periodic breathers  can have  large, small, and mixed frequency-magnitudes.
\end{theorem}

The conclusion of Theorem \ref{TH22} may seem surprising at first glance, but it is mathematically natural. The coupling structure of \eqref{MACXIT2} allows for the correlation between multiple oscillators, the number of which can tend to infinity asymptotically. However, for each fixed $n$, these couplings are indeed \textit{finite}. This aligns with the infinite-dimensional non-resonance condition and the spatial structure proposed by Bourgain \cite{MR2180074} (see also Section \ref{SEC1}). Therefore, as long as the perturbation is sufficiently small, the perturbed system can preserve most of the dynamics of the unperturbed system.

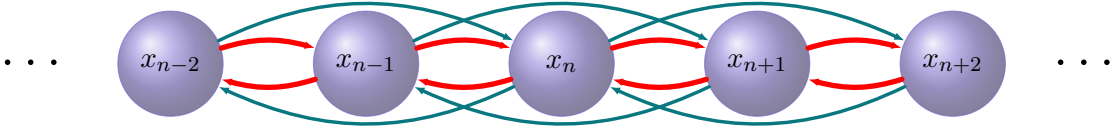
\begin{figure}[H]
	\centering
	\begin{tikzpicture}[
		arr/.style={-{Latex[length=4pt,bend]}, line cap=round}
		]
		
		\pgfmathsetmacro{\dx}{2.6}   
		\pgfmathsetmacro{\wNear}{2.0pt}
		\pgfmathsetmacro{\wMid}{1.3pt}
		\pgfmathsetmacro{\wFar}{0.7pt}
		\pgfmathsetmacro{\bend}{6mm}
		
		\pgfmathsetmacro{\leftB}{-1.3}
		\pgfmathsetmacro{\rightB}{4*\dx+1.3}
		\pgfmathsetmacro{\shift}{-(\leftB+\rightB)/2}  
		\begin{scope}[shift={(\shift,0)}]
			
			\foreach \idx [count=\i from 0] in {n-2,n-1,n,n+1,n+2}{
				\begin{scope}[opacity=0.58]
					\node[ball3Donly] (ball\i) at (\i*\dx,0) {};
				\end{scope}
				\node[text=black] at (\i*\dx,0) {$\boldsymbol{x_{\idx}}$};
			}
			
			\node[scale=1.8, black] at (-1.8,0) {$\cdots$};
			\node[scale=1.8, black] at (4*\dx+1.8,0) {$\cdots$};
			
			\draw[arr,red,line width=\wNear] (ball0) to[bend left=\bend] (ball1);
			\draw[arr,red,line width=\wNear] (ball1) to[bend left=\bend] (ball0);
			\draw[arr,red,line width=\wNear] (ball1) to[bend left=\bend] (ball2);
			\draw[arr,red,line width=\wNear] (ball2) to[bend left=\bend] (ball1);
			\draw[arr,red,line width=\wNear] (ball2) to[bend left=\bend] (ball3);
			\draw[arr,red,line width=\wNear] (ball3) to[bend left=\bend] (ball2);
			\draw[arr,red,line width=\wNear] (ball3) to[bend left=\bend] (ball4);
			\draw[arr,red,line width=\wNear] (ball4) to[bend left=\bend] (ball3);
			
			\draw[arr,myTeal,line width=\wMid] (ball0) to[bend left=\bend+3mm] (ball2);
			\draw[arr,myTeal,line width=\wMid] (ball2) to[bend left=\bend+3mm] (ball0);
			
			\draw[arr,myTeal,line width=\wMid] (ball1) to[bend left=\bend+3mm] (ball3);
			\draw[arr,myTeal,line width=\wMid] (ball3) to[bend left=\bend+3mm] (ball1);
			
			\draw[arr,myTeal,line width=\wMid] (ball2) to[bend left=\bend+3mm] (ball4);
			\draw[arr,myTeal,line width=\wMid] (ball4) to[bend left=\bend+3mm] (ball2);
			
		\end{scope}
	\end{tikzpicture}
	\caption{General coupling in system \eqref{MACXIT2}: coupling occurs exclusively among each consecutive triplet of oscillators.}
	\label{FIG2}
\end{figure}

\begin{figure}[H]
	\centering
	\begin{tikzpicture}[
		arr/.style={-{Latex[length=4pt,bend]}, line cap=round}
		]
		
		\pgfmathsetmacro{\dx}{2.6}   
		\pgfmathsetmacro{\wNear}{2.0pt}
		\pgfmathsetmacro{\wMid}{1.3pt}
		\pgfmathsetmacro{\wFar}{0.7pt}
		\pgfmathsetmacro{\bend}{6mm}
		
		\pgfmathsetmacro{\leftB}{-1.3}
		\pgfmathsetmacro{\rightB}{4*\dx+1.3}
		\pgfmathsetmacro{\shift}{-(\leftB+\rightB)/2}  
		\begin{scope}[shift={(\shift,0)}]
			
			\foreach \idx [count=\i from 0] in {n-2,n-1,n,n+1,n+2}{
				\begin{scope}[opacity=0.58]
					\node[ball3Donly] (ball\i) at (\i*\dx,0) {};
				\end{scope}
				\node[text=black] at (\i*\dx,0) {$\boldsymbol{x_{\idx}}$};
			}
			
			\node[scale=1.8, black] at (-1.8,0) {$\cdots$};
			\node[scale=1.8, black] at (4*\dx+1.8,0) {$\cdots$};
			
			\draw[arr,red,line width=\wNear] (ball0) to[bend left=\bend] (ball1);
			\draw[arr,red,line width=\wNear] (ball1) to[bend left=\bend] (ball0);
			\draw[arr,red,line width=\wNear] (ball1) to[bend left=\bend] (ball2);
			\draw[arr,red,line width=\wNear] (ball2) to[bend left=\bend] (ball1);
			\draw[arr,red,line width=\wNear] (ball2) to[bend left=\bend] (ball3);
			\draw[arr,red,line width=\wNear] (ball3) to[bend left=\bend] (ball2);
			\draw[arr,red,line width=\wNear] (ball3) to[bend left=\bend] (ball4);
			\draw[arr,red,line width=\wNear] (ball4) to[bend left=\bend] (ball3);
			
			\draw[arr,myTeal,line width=\wMid] (ball0) to[bend left=\bend+3mm] (ball2);
			\draw[arr,myTeal,line width=\wMid] (ball2) to[bend left=\bend+3mm] (ball0);
			
			\draw[arr,myTeal,line width=\wMid] (ball1) to[bend left=\bend+3mm] (ball3);
			\draw[arr,myTeal,line width=\wMid] (ball3) to[bend left=\bend+3mm] (ball1);
			
			\draw[arr,myTeal,line width=\wMid] (ball2) to[bend left=\bend+3mm] (ball4);
			\draw[arr,myTeal,line width=\wMid] (ball4) to[bend left=\bend+3mm] (ball2);
			\draw[arr,myOrange,line width=\wFar] (ball0) to[bend left=\bend+6mm] (ball3);
			\draw[arr,myOrange,line width=\wFar] (ball3) to[bend left=\bend+6mm] (ball0);
			
			\draw[arr,myOrange,line width=\wFar] (ball1) to[bend left=\bend+6mm] (ball4);
			\draw[arr,myOrange,line width=\wFar] (ball4) to[bend left=\bend+6mm] (ball1);
		\end{scope}
	\end{tikzpicture}
	\caption{General coupling in system \eqref{MACXIT2}: coupling occurs exclusively among each consecutive quartet of oscillators.}
	\label{FIG3}
\end{figure}
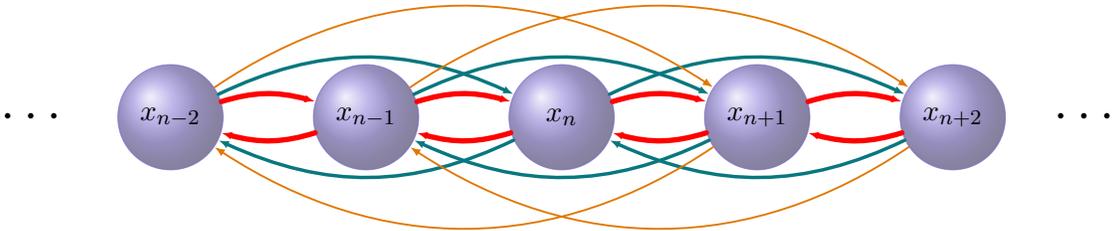
To facilitate an intuitive understanding of Theorem \ref{TH22}, Figures \ref{FIG2} and \ref{FIG3} depict two elementary coupling configurations: nearest-triplet and nearest-quartet interactions, respectively. The coupling strength attenuates with inter-oscillator distance, a decay represented by the progressively decreasing thickness of the arrows. While these configurations constitute strict subclasses of the general coupling prescribed in Theorem \ref{TH22}, they serve as finite prototypes from which the  reader may infer the limiting behavior as the interaction range tends to infinity, thereby elucidating the general result via finite-to-asymptotic infinite extrapolation.

\subsection{Special cases and asymptotic analysis} 
In order to ensure the validity of the assumptions for specific cases,  we need to quantitatively determine  the relevant terms in the Hamiltonian \eqref{MACHA3}, using  asymptotic analysis for $ H_0 (h) $ (or $ \rho (h)$) near $ 0 ^+$. Note that, from \eqref{MACrho}, for $0< \rho\ll1 $, we can write $ \rho(h) $  as 
\[\rho \left( h \right) = 2\sqrt 2 \int_0^{{V^{ - 1}}\left( h \right)} {\sqrt {h - V\left( s \right)} {\rm d}s}. \]
Next, we present two common examples to more clearly illustrate this point to the  reader.

\subsubsection{The case of $ V(x)=x^{2p} $ with $ p \in \mathbb{N}^+ $}
 We first mention that, except for $ p=1 $, the other cases do not satisfy the assumptions in \cite{MR1889993} ($ V'(0)=0 $ and $ V''(0)\ne0 $ in Assumption (V1) there). As a consequence, \cite[Corollary 1.2]{MR1889993} cannot be directly obtained from the KAM theorem there.  However, our analysis does work.
 
 A straightforward calculation shows that
\[\rho \left( h \right)|_{{p = 1}} = 2\sqrt 2 \int_0^{{h^{\frac{1}{{2p}}}}} {\sqrt {h - {s^{2p}}} {\rm d}s}  = \left( {\frac{{\sqrt 2 }}{p}\int_0^1 {{z^{\frac{1}{{2p}} - 1}}{{\left( {1 - z} \right)}^{\frac{3}{2} - 1}}{\rm d}z} } \right){h^{\frac{{1 + p}}{{2p}}}} = \frac{{\sqrt 2 }}{p}\mathrm{B}\left( {\frac{1}{{2p}},\frac{3}{2}} \right){h^{\frac{{1 + p}}{{2p}}}},\]
where $ \mathrm{B}(x,y) $ denotes the standard Beta function. Therefore,
\[{\rho }\left( h \right)|_{{p = 1}} = {\mathcal{O}^\# }\left( h \right),\quad {\rho }\left( h \right)|_{{p = 2}} = {\mathcal{O}^\# }\left( {{h^{\frac{3}{4}}}} \right),\quad  \cdots\cdots\quad , \]
and this leads to 
\[{H_{0}}\left( h \right) |_{{p = 1}}= {\mathcal{O}^\# }\left( h \right),\quad {H_{0}}\left( h \right)|_{{p = 2}} = {\mathcal{O}^\# }\left( {{h^{\frac{3}{4}}}} \right),\quad  \cdots\cdots \]
Here, $\mathcal{O}^\#$ denotes the same order up to a non-zero constant; this holds for all $h\in\mathcal{I}$, not only for $0<h\ll 1$, because both $\rho(h)|_p$ and $H_0(h)|_p$ are simple power functions. In other words, all of our assumptions can hold because the compactness of the interval $ \mathcal{I}\subset \mathbb{R}^+ $ required in Assumption \ref{con:a3} implies that it cannot include $ h = 0 $.  Hence, we obtain frequency-preserving almost periodic breathers for the Aubry--MacKay conjecture. 

 However, for general $ V(x) $, we have to estimate the asymptotic expansions for $ \rho $ and $ H_0 $.

\subsubsection{The case of $ V\left( x \right) = 2{x^2} + {x^4} $}
 In this case, $ V(0)=V'(0)=0$, and $ V''(0)>0 $. Below, we assume that the interval $ \mathcal{I}\subset \mathbb{R}^+ $ in Assumption \ref{con:a3} is sufficiently close to $ h=0 $. A straightforward calculation gives
\begin{align*}
	\rho \left( h \right) &= 2\sqrt 2 \int_0^{{V^{ - 1}}\left( h \right)} {\sqrt {h - V\left( s \right)} {\rm d}s} \\
	& = \frac{{\sqrt 2 }}{2}{h^{\frac{3}{2}}}\int_0^1 {\frac{{\sqrt {1 - z} }}{{\sqrt {\sqrt {hz + 1}  - 1} \sqrt {hz + 1} }}{\rm d}z} \\
	& = \frac{{\sqrt 2 }}{2}{h^{\frac{1}{2}}}\int_0^1 {\frac{{\sqrt {1 - z} }}{z}\left( {1 + \frac{1}{{\sqrt {hz + 1} }}} \right)\sqrt {\sqrt {hz + 1}  - 1} {\rm d}z} .
\end{align*}
With the asymptotic expansion
\[\left( {1 + \frac{1}{{\sqrt {a + 1} }}} \right)\sqrt {\sqrt {a + 1}  - 1}  = \sqrt 2 {a^{\frac{1}{2}}} - \frac{3}{{4\sqrt 2 }}{a^{\frac{3}{2}}} + \frac{{35}}{{64\sqrt 2 }}{a^{\frac{5}{2}}} +  \cdots  + {\mathcal{O}^\# }\left( {{a^{ {N + \frac{1}{2}} }}} \right) +  \cdots ,\quad  a\to 0^+,\]
we obtain the analytic function $ \rho(h) $ on $ \mathcal{I} $ as
\[\rho \left( h \right) = {A_1}h + {A_2}{h^2} +  \cdots  = \sum\limits_{j = 1}^\infty  {{A_j}{h^j}} , \quad A_1,A_2\ne0.\]
By applying  the Lagrange inversion theorem, we obtain the analytic function $ H_0(h) $ on $ \mathcal{I} $ (which might be smaller) as 
\[{H_0}\left( h \right) = {\rho ^{ - 1}}\left( h \right) = \frac{1}{{{A_1}}}h - \frac{{{A_2}}}{{A_1^3}}{h^2} +  \cdots  = \sum\limits_{j = 1}^\infty  {{B_j}{h^j}} ,\quad B_1,B_2\ne0.\]
Therefore, for any parameter $ \xi $ satisfying $ \sum\nolimits_{n \in \mathbb{Z}}  {\left| {{\xi _n}} \right|}  \ll 1$ (e.g., $ <|A_2A_1^{-3}| $), we have
\[\left| {e\left( \xi  \right)} \right| \leqslant \sum\limits_{n \in \mathbb{Z}}  {\left| {{H_0}\left( {{\xi _n}} \right)} \right|}  = \mathcal{O}\left( {\sum\limits_{n \in \mathbb{Z}}  {\left| {{\xi _n}} \right|} } \right) =\mathcal{O}\left( {\left| \xi \right|} \right),\]
\[\left| {\left\langle {\omega \left( \xi  \right),I} \right\rangle } \right| \leqslant \sum\limits_{n \in \mathbb{Z}}  {\left| {{H_0'}\left( {{\xi _n}} \right)} \right|\left| {{I_n}} \right|}  = \mathcal{O}\left( {\sum\limits_{n \in \mathbb{Z}} {\left| {{\xi _n}} \right|\left| {{\xi_n}} \right|} } \right) = \mathcal{O}\left( {\left| \xi \right|} \right),\]
and
\[{H_0''}\left( {{\xi _n}} \right) =  - 2{A_2}A_1^{ - 3} + \mathcal{O}\left( {\left| {{\xi _n}} \right|} \right) =  - 2{A_2}A_1^{ - 3} + o\left( 1 \right).\]
These imply that the KAM Theorem \ref{FULLT1} does work, whenever we choose some $ \xi $ near $ 0 $ such that $ \omega(\xi) $ is Diophantine or weak Diophantine (note that such frequencies admit   full measure). Hence, we obtain frequency-preserving almost periodic breathers for the Aubry--MacKay conjecture. 

\subsection{Some notes on  Assumptions \ref{con:a2}--\ref{con:a4}}

	\begin{enumerate}[label=(N\arabic*), ref=(N\arabic*)]
\item 	The Lagrange inversion theorem gives the power expansion (into $y$) for an inverse $x=f^{-1}(y)$ of an analytic function $f(x)=y$. However, it requires that $f'(0)\ne 0$. In our case, $V'(0)=0$ must hold; otherwise, the closed curve $\Gamma(h)$ in Assumption \ref{con:a3} does not exist. \textit{As a consequence, we cannot directly use the Lagrange inversion theorem to calculate $\rho(h)$ from the integral involving $V(x)$.}

\item The analyticity of $H_0=\rho^{-1}$ in Assumption \ref{con:a4} is essential. Without it, $H_0$ might be nowhere analytic, even though it is $C^\infty$. Note that the analyticity of $V(x)$ in Assumption \ref{con:a1} does not necessarily ensure the analyticity of $H_0$, even if we further assume Assumption \ref{con:a2}. Indeed, we can construct $C^\infty$ counterexamples with $V(x)=x^2+x^4+\widetilde{V}(x)$, where $ \widetilde{V}(x) $ is a term that is  $ C^\infty $  but nowhere analytic. As a consequence, we cannot apply any analytic KAM theorems to obtain almost periodic breathers. 

\item  A stronger but easier-to-verify version of Assumption \ref{con:a4} is:\vspace{2mm}\\
(A4$ ' $) Let $ \rho=\rho(h) $ be the area of the region enclosed by the closed curve $ \Gamma(h) $, i.e., 
\[\rho \left( h \right) = \oint_{\frac{1}{2}y^2 + V\left( x \right) = h} {y{\rm d}x} .\]
We assume that under the potential $ V $, for any $ h \in \overline{\mathcal{I}} $, $ \rho(h) $ is analytic,  $ \rho '\left( h \right) \ne 0 $ and $\rho ''\left( h \right) \ne 0 $.\vspace{2mm}\\
Note that by assuming the analyticity of $ \rho(h) $ on $ \overline{\mathcal{I}} $, we can obtain from the Lagrange inversion theorem the analyticity of $ H_0(h)=\rho^{-1}(h) $ on $ \overline{\mathcal{I}} $ (or on a smaller set). Moreover, the first term and the second term in $ H_0(h) $ (with respect to $ h $) are nondegenerate, which is sufficient for the nondegeneracy in our KAM Theorems \ref{FULLT1} and \ref{FULLT2}.
\end{enumerate}

\section{Comparison with previous literature}\label{SECNEW}
The primary motivations and main contributions of this paper have been thoroughly detailed in the Introduction. To facilitate a deeper understanding, we devote this section to comparing our main results with several pertinent prior works.

\subsection{Comparison with P\"oschel's KAM}
In 1990, P\"oschel \cite{MR1037110} established the existence of full-dimensional invariant tori for infinite-dimensional Hamiltonian systems under very  general non-resonance conditions and spatial structures. The Hamiltonian considered therein, however, lacks the Legendre-type nondegeneracy condition imposed in the present paper; consequently, the frequencies are subject to drift, marking a fundamental distinction. Furthermore, the overly broad nature of such general non-resonance conditions and spatial structures renders their application to concrete physical problems somewhat abstract and intractable.
In contrast, the present paper establishes infinite-dimensional frequency-preserving KAM theorems within a Bourgain-type (and more general) framework. Such results are intrinsically better suited for applications to infinite-dimensional networks (lattice systems). It must be emphasized, however, that this does not preclude the applicability of the results in \cite{MR1037110}; parallel to the analysis in Section \ref{SECAMC}, one can still deduce the existence of a family of almost periodic breathers for infinite-dimensional networks. As elucidated by Yuan \cite{MR1889993}, KAM theorems for infinite-dimensional Hamiltonian systems are capable of achieving this. Nevertheless, in the absence of a Kolmogorov-type KAM theorem concerning frequency preservation, \textit{it remains an open question whether frequency-preserving almost periodic breathers exist,} which serves as a primary motivation for our study. From a technical standpoint, the approach in \cite{MR1037110} diverges substantially from ours. The present paper combines the ideas as well as techniques of Salamon \cite{MR2111297} and Bourgain \cite{MR2180074}, departing from the methodology utilized in \cite{MR1037110}. 

In conclusion, applying the results of \cite{MR1037110} to infinite-dimensional networks yields conclusions fundamentally distinct from ours. From a dynamical systems perspective, \cite{MR1037110} cannot address whether the original breather frequencies persist invariantly under perturbation. Furthermore, the underlying techniques employed in the two works are substantially different.

\subsection{Comparison with Corsi--Gentile--Procesi's KAM}
 Corsi--Gentile--Procesi \cite{MR4781767} recently investigated the existence of full-dimensional tori for an analytic mechanical system
\begin{equation}\label{CGPH}
	H\left( {x,y} \right) = \frac{1}{2}y \cdot y + \varepsilon P\left( x \right)
\end{equation}
based on the tree formalism.  It is worth noting that while both the system considered therein and our Hamiltonian are analytic, our framework is more general. Consequently, in a certain sense, the present paper partially resolves the open problem posed in \cite[Section 3]{MR4781767}, albeit under a slightly different problem setting. As noted in Remark \ref{FULLRe1,1}, our Hamiltonian formulation permits the unperturbed system to be \textit{non-integrable} and allows for a \textit{more general} perturbation; more precisely, both may depend on the angular variable $x$, a feature absent in \cite{MR4781767}. This distinction manifests most clearly in applications to concrete physical models. For instance, when considering the network \eqref{MACXT} in Section \ref{SECAMC}, a reduction procedure yields the Hamiltonian in \eqref{MACHA3}, wherein the perturbation intrinsically depends on the angular variable $x$. As a result, the frequency-preserving KAM theorems established in this paper are applicable, whereas the results from \cite{MR4781767} are not directly applicable.  However, we emphasize that, as stated in \cite[Section 3]{MR4781767}, the motivation for their results primarily stems from the study of Hamiltonian PDEs, wherein formulations such as \eqref{CGPH} arise quite naturally. In contrast, the results of the present paper are designed for investigating \textit{infinite-dimensional networks (lattice systems)}. Furthermore, we place greater emphasis on the existence of \textit{frequency-preserving} infinite-dimensional tori, a viewpoint that is inherently more natural from a dynamical systems perspective. Indeed, the results herein appear not to be suited for Hamiltonian PDEs, given that the corresponding Hamiltonians are usually expanded around an elliptic equilibrium (the zero solution), and hence the size of the nonlinearity and the room for modulation cannot be chosen independently but are of the same order---a characteristic that diverges from the techniques employed in this paper\footnote{Nevertheless, provided that specific PDEs can be suitably reduced to the Hamiltonian form of Section \ref{SECAMC} (for instance, in the presence of additional parameters), our KAM results become applicable. Note that we do not impose any spectral asymptotic conditions.}.

We also mention that \cite{MR4781767} does not use the truncation estimate based on $ |k|_\eta $ as in  \cite[Lemma B.2]{MR4201442} (see also \cite[Lemma 5.7]{arXiv:2306.08211}).  Instead, it uses 
\[|k|_{\star}:=\sum_{j \in \mathbb{Z}} \left|k_j\right|h_j,  \quad  h_j=h_{-j} \in \mathbb{R}_{+} \forall j \in \mathbb{Z},  \quad  h_{j+1} \geqslant h_j \forall j \in \mathbb{Z}_{+} \]
with 
\[\limsup _{j \rightarrow+\infty} \frac{(\log (1+\langle j\rangle))^\sigma}{h_j}<+\infty \; \text { for some } \sigma>2\]
to obtain a small divisor estimate as $ N \to +\infty $ for any $ \mu_1,\mu_2>0 $:
\begin{equation}\label{VECC}
	\sup _{\substack{\nu \in \mathbb{Z}_*^\infty},\; |k|_{\star} \leqslant N} \prod_{j \in \mathbb{Z}}\left(1+\langle j\rangle^{\mu_1}\left|\nu_j\right|^{\mu_2}\right) \leqslant K_1 \exp\left( \frac{K_2 N}{(\log N)^{\sigma-1}}\right)\;\text{for some universal\;} K_1, K_2>0,
\end{equation}
thereby maximally relaxing the spatial structure of the analytic perturbation $ P $ with respect to $ k $. This is of particular interest because the small divisor estimate given in \eqref{VECC} is critical in the KAM iteration process, from the perspective of the Brjuno condition or the weak Diophantine condition (see, for instance, P\"oschel \cite{MR1037110} and the authors \cite{MR4836959}).  Similarly, we also address this issue in Comments \ref{com:c4} and \ref{com:c5}, where we consider the more critical case. Instead of employing truncation estimates like \eqref{VECC}, we directly utilize the form of the control function provided in Definition \ref{weakdio}; however, we believe that the underlying mechanism remains the same.   \vspace{3mm}

\subsection{Comparison with the authors' recent KAM}    
	As pointed out in Comment \ref{com:c1}, the authors  \cite{arXiv:2306.08211} recently established an infinite-dimensional KAM theorem allowing for $C^\infty$ regularity (with a corresponding spatial structure) inspired by P\"oschel's work \cite{PoarXiv}. Our KAM results differ from this work in several fundamental aspects. First,  \cite{arXiv:2306.08211} concerns infinite-dimensional vector fields independent of the action variable $y$, whereas the present paper deals with infinite-dimensional Hamiltonian systems. Second, due to the lack of a nondegeneracy condition, the surviving KAM invariant tori in \cite{arXiv:2306.08211} must exhibit frequency drift. In contrast, our approach ensures exact frequency preservation by virtue of a Legendre-type nondegeneracy condition. Most crucially,  the result in \cite{arXiv:2306.08211} accommodates $C^\infty$ regularity, while our present KAM scheme is strictly confined to the analytic category. Let us elaborate on this distinction. The approach in \cite{arXiv:2306.08211} adapts the strategy of P\"oschel \cite{PoarXiv}, combining KAM coordinate transformations with an analytic smoothing technique to construct an $m$-weighted norm. By utilizing a Bourgain-type (or more general) framework as an intermediary, it deduces the $m$-weighted regularity  \textit{a posteriori}, thereby accommodating non-analytic $C^\infty$ vector fields.  At the crux of this approach, the authors introduced a \textit{balancing sequence} to balance the non-resonance condition against the regularity. In the absence of such a sequence, the associated KAM iteration could not be established. It is also worth noting that non-analytic functions \textit{cannot} be evaluated pointwise on arbitrary infinite-dimensional thickened tori, making the previous approximation scheme indispensable. Consequently, the KAM technique employed in the present paper is fundamentally distinct from the approach therein, and cannot be directly extended to accommodate regularities weaker than analyticity.	The sole connection between \cite{arXiv:2306.08211} and the present paper lies in the discussion on the spatial structure in Section \ref{SEC24}. Nevertheless, rather than being a primary novelty of our KAM results, this simply provides a feasible extension. 
	
\textit{In conclusion, \cite{arXiv:2306.08211} and the present paper have essentially no overlap in terms of the systems studied, motivations, applications, and underlying techniques.}

\section{Proof of Main Result I}\label{SEC4}
\subsection{Proof of Theorem \ref{FULLT1}: KAM via the infinite-dimensional Diophantine non-resonance condition}\label{FULLSEC3}

\subsubsection{Some preliminary lemmas}\label{FUUSECPRE}
Here, we first provide  four basic lemmas without proof as a foundation. Detailed proofs can be found, for instance, in  \cite[Lemmas 2.5, 2.6, 2.7, and 2.11]{MR4201442}. The first three aspects, namely  the Banach algebraic property for the space $ \mathcal{G}\left( {\mathbb{T}_\sigma ^\infty } \right)  $, the definition of the Fourier constant, and the Cauchy's estimate associated with the higher order  derivatives, exhibit similarities to the finite-dimensional case. The final aspect, namely  the homological equation with a Diophantine frequency belonging to $ {\mathcal{D}_{\gamma ,\mu }} $, differs significantly from the finite-dimensional case. This is because the finite-dimensional Diophantine non-resonance condition is characterized by some finite-order polynomial, and consequently, the coefficient of control in the homological equation there also exhibits polynomial characteristics (not the exponential type in Lemma \ref{FULLtdfcyl} for the infinite-dimensional case), as illustrated in \cite[Lemma 2]{MR2111297}.

\begin{lemma}[Banach algebraic property]
	For $ u,v \in \mathcal{G}\left( {\mathbb{T}_\sigma ^\infty } \right) $ with $ \sigma>0 $, we have $ uv \in \mathcal{G}\left( {\mathbb{T}_\sigma ^\infty } \right) $ and $ \left\|uv\right\|_{\sigma } \leqslant \left\|u\right\|_{\sigma }\left\|v\right\|_{\sigma } $.
\end{lemma}

\begin{lemma}[Fourier constant]
	For $ u \in \mathcal{G}\left( {\mathbb{T}_\sigma ^\infty } \right) $, we have
	\[\widehat u_0=\int_{{\mathbb{T}^\infty }} {u ( x  ){\rm d}x} : = \mathop {\lim }\limits_{N \to  + \infty } \frac{1}{{{{\left( {2\pi } \right)}^{2N+1}}}}\int_{{\mathbb{T}^{2N+1}}} {u ( x  ){\rm d}{x_{-N}} \cdots {\rm d}{x_N}}  .\]
	Moreover, for any $ 0 \ne \ell \in \mathbb{Z}_ * ^\infty  $, we have
	\[\widehat u_\ell = \int_{{\mathbb{T}^\infty }} u (x){{\rm e}^{ - {\rm{i}}\left\langle {\ell ,x} \right\rangle}}{\rm d}x = \mathop {\lim }\limits_{N \to +\infty } \frac{1}{{{{(2\pi )}^{2N + 1}}}}\int_{{\mathbb{T}^{2N + 1}}} u (x){{\rm e}^{ - {\rm{i}}\left\langle {\ell ,x} \right\rangle}} {\rm d}{x_{ - N}}\cdots  {\rm d}{x_N}.\]
\end{lemma}

\begin{lemma}[Cauchy's estimate]\label{FULLcauchy}
	Let $ u \in \mathcal{G}\left( {\mathbb{T}_{\sigma  + \rho }^\infty } \right) $ with $ \sigma ,\rho  > 0 $. Then for any $ k \in \mathbb{N} $, the $ k $-th derivative $ D_x^ku $ satisfies the following estimate with $ c\left( k \right) > 0 $  depending only on the order $ k $:
	\[\left\|D_x^ku\right\|_{\sigma } \leqslant c\left( k \right){\rho ^{ - k}}\left\|u\right\|_{{\sigma  + \rho }}.\]
\end{lemma}

\begin{lemma}[Diophantine homological equation]\label{FULLtdfcyl}
	Let $ \mu,\eta,\sigma ,\rho  > 0 $ and a Diophantine frequency $ \omega \in {\mathcal{D}_{\gamma ,\mu }} $ be given. Then there exists a constant $ \tau  = \tau \left( {\eta ,\mu } \right) > 0 $ such that for every $ g \in \mathcal{G}_0\left( {\mathbb{T}_{\sigma  + \rho }^\infty } \right) $,
	the homological equation \eqref{FULLtdfc}	admits a unique solution $ f \in \mathcal{G}_0\left( {\mathbb{T}_\sigma ^\infty } \right) $, and
	\begin{equation}\notag
		\left\|f\right\|_{\sigma } \leqslant \exp \left( {\frac{\tau }{{{\rho ^{\frac{1}{\eta} }}}}\log \left( {\frac{\tau }{\rho }} \right)} \right)\left\|g\right\|_{{\sigma  + \rho }}.
	\end{equation}
\end{lemma}

Next, we establish a  basic lemma for Taylor's estimates, which will be used in the subsequent proof.

\begin{lemma}[Taylor's estimates]\label{FULLtaylor}
	Let $ \sigma,M>0 $ be given. Then for a complex function $ \mathcal{H}\left( {x,y} \right) $ on $ {\mathscr{D}_{\sigma,\sigma}} $ with $ \left\|{\mathcal{H}_{yy}}\left( {x,y} \right)\right\|_{{\sigma ,\sigma}} \leqslant M $,  we have
	\begin{align}
		\label{FULLH1}	&\left\|\mathcal{H}\left( {x,y} \right) - \mathcal{H}\left( {x,0} \right) -  \left\langle {{\mathcal{H}_y}\left( {x,0} \right)},y \right\rangle   \right\|_{\sigma } \leqslant M\left\|y\right\|_\sigma ^2,\\
		\label{FULLH2}&\left\|{\mathcal{H}_y}\left( {x,y} \right) - {\mathcal{H}_y}\left( {x,0} \right)\right\|_{\sigma } \leqslant M\left\|y\right\|_{\sigma },\\
		\label{FULLH3}&\left\|{\mathcal{H}_y}\left( {x,y} \right) - {\mathcal{H}_y}\left( {x,0} \right) -\left\langle {\mathcal{H}_{yy}}\left( {x,0} \right),y\right\rangle\right\|_{\sigma } \leqslant \frac{{M\left\|y\right\|_\sigma ^2}}{{\sigma - \left\|y\right\|_{\sigma }}}.
	\end{align}
\end{lemma}
\begin{proof}
	According to Taylor's formula, we have
	\begin{align*}
		\left\|\mathcal{H}\left( {x,y} \right) - \mathcal{H}\left( {x,0} \right) -  \left\langle{{\mathcal{H}_y}\left( {x,0} \right),y}\right\rangle  \right\|_{\sigma } &= \left\|\int_0^1 {\int_0^t { \left\langle {\left\langle {{\mathcal{H}_{yy}}\left( {x,sy} \right),y} \right\rangle },{y } \right\rangle   {\rm d}s{\rm d}t} } \right\|_{\sigma }\\
		& \leqslant \int_0^1 {\int_0^t {\left\| \left\langle {\left\langle {{\mathcal{H}_{yy}}\left( {x,sy} \right),y} \right\rangle },{y } \right\rangle  \right\|_{\sigma }{\rm d}s{\rm d}t} }  \\
		&\leqslant \int_0^1 {\int_0^t {\mathop {\sup }\limits_{i\in \mathbb{Z}} \left\|\sum\limits_{j \in \mathbb{Z}} {\mathcal{H}_{yy}^{\left( {i,j} \right)}\left( {x,sy} \right){y_i}{y_j}} \right\|_{\sigma }{\rm d}s{\rm d}t} } \\
		& \leqslant M\mathop {\sup }\limits_{j\in \mathbb{Z}} \left\|{y_j}\right\|_{\sigma } \cdot \mathop {\sup }\limits_{i \in \mathbb{Z}} \left\|{y_i}\right\|_{\sigma } \\
		&\leqslant M\left\|y\right\|_\sigma ^2,
	\end{align*}
	which proves \eqref{FULLH1}. Regarding \eqref{FULLH2}, we have
	\begin{align*}
		\left\|{\mathcal{H}_y}\left( {x,y} \right) - {\mathcal{H}_y}\left( {x,0} \right)\right\|_{\sigma } &= \left\|\int_0^1 {\left\langle {{\mathcal{H}_{yy}}\left( {x,ty} \right),y} \right\rangle {\rm d}t} \right\|_{\sigma } \leqslant \int_0^1 {\left\|\left\langle {{\mathcal{H}_{yy}}\left( {x,ty} \right),y} \right\rangle \right\|_{\sigma }{\rm d}t} \\
		&= \int_0^1 {\mathop {\sup }\limits_{i,j \in \mathbb{Z}} \left\|\mathcal{H}_{yy}^{\left( {i,j} \right)}\left( {x,ty} \right){y_j}\right\|_{\sigma }{\rm d}t}  \leqslant M\mathop {\sup }\limits_{j \in \mathbb{Z}} \left\|{y_j}\right\|_{\sigma } \leqslant M \left\|y\right\|_{\sigma }.
	\end{align*}
	Finally, by considering the curve  $ \Gamma : = \left\{ {\lambda  \in \mathbb{C}: \left| \lambda  \right| = \sigma\left\|y\right\|_{\sigma }^{-1} > 1} \right\} $, we obtain that
	\begin{align*}
		\left\|{\mathcal{H}_y}\left( {x,y} \right) - {\mathcal{H}_y}\left( {x,0} \right) -\left\langle {\mathcal{H}_{yy}}\left( {x,0} \right),y\right\rangle\right\|_{\sigma } &= \left\|\int_0^1 {\left( {\left\langle{\mathcal{H}_{yy}}\left( {x,ty} \right),y\right\rangle - \left\langle{\mathcal{H}_{yy}}\left( {x,0} \right),y\right\rangle} \right){\rm d}t} \right\|_{\sigma }\\
		& = \left\|\int_0^1 {\frac{1}{{2\pi {\rm i}}}\int_\Gamma  {\frac{1}{{\lambda \left( {\lambda  - 1} \right)}}} \left\langle{\mathcal{H}_{yy}}\left( {x,\lambda ty} \right),y\right\rangle{\rm d}\lambda {\rm d}t} \right\|_{\sigma } \\
		& \leqslant \int_0^1 {\frac{1}{{2\pi }} \cdot \left| \Gamma  \right| \cdot \frac{1}{{\left| \lambda  \right|\left( {\left| \lambda  \right| - 1} \right)}} \cdot \mathop {\sup }\limits_{i,j\in \mathbb{Z}} \left\|\mathcal{H}_{yy}^{\left( {i,j} \right)}\left( {x,\lambda ty} \right){y_j}\right\|_{\sigma }{\rm d}t} \\
		& \leqslant \frac{1}{{2\pi }} \cdot 2\pi \frac{\sigma}{{\left\|{y}\right\|_{\sigma }}} \cdot \frac{1}{{\left| \lambda  \right|\left( {\left| \lambda  \right| - 1} \right)}} \cdot M\mathop {\sup }\limits_{j\in \mathbb{Z}} \left\|{y_j}\right\|_{\sigma }\\
		& \leqslant   \frac{{M\left\|y\right\|_\sigma ^2}}{{\sigma - \left\|y\right\|_{\sigma }}},
	\end{align*}
	which proves \eqref{FULLH3}.
\end{proof}

\subsubsection{Proof of Theorem \ref{FULLT1}}
We now prove Theorem \ref{FULLT1}. Without loss of generality, we will introduce some universal constants that are independent of the iterative process and may vary in different contexts. In particular, if necessary, we will indicate which variables they depend on. We also emphasize that after balancing spatial structure, regularity, and the non-resonance condition, KAM analysis in the infinite-dimensional context is indeed similar to that in the finite-dimensional context. For further insights on this aspect, refer to P\"oschel \cite{MR1037110}, the authors \cite{arXiv:2306.08211}, and the references therein.

For any fixed $ {2^{ - \eta }} < q < 2^{-1}\left({2^{ - \eta }}+1\right) $, let us define the contraction sequence as
\[{\sigma _\nu }: =  \frac{1}{2}\sigma \left( {1 + {q^\nu }} \right),\quad \nu  \in \mathbb{N}.\]
It is straightforward to verify that $ {\sigma _0} = \sigma,{\sigma _\infty } = \sigma /2 $, and 
\[{\sigma _\nu } - {\sigma _{\nu  + 1}} = \frac{1}{2}\sigma \left( {1 - q } \right){q^\nu },\quad \nu  \in \mathbb{N}.\]
As we will see later, the appropriately selected contraction sequence ensures that our KAM iteration is super-exponentially convergent.

To obtain the desired frequency-preserving, symplectic and analytic transformations in the KAM iteration, we construct, for $x \in \mathbb{T}_{\sigma_\nu}^\infty$ in the $\nu$-th step, the following partial differential equations using the generating function method of Kolmogorov \cite{MR0068687} and Salamon \cite{MR2111297}:
\begin{align}
	\label{FULLgfe1}\omega  \cdot {\partial _x}a\left( x \right) &= \int_{{\mathbb{T}^\infty }} {{{\mathscr{H}}^\nu }\left( {\xi ,0} \right){\rm d}\xi }  - {{\mathscr{H}}^\nu }\left( {x,0} \right),\\
\label{FULLgfe2}	\omega  &= \int_{{\mathbb{T}^\infty }} {\left( {{\mathscr{H}}_y^\nu \left( {\xi ,0} \right) + \left\langle{\mathscr{H}}_{yy}^\nu \left( {\xi ,0} \right), {\alpha  + {a_x}\left( \xi  \right)}\right\rangle} \right){\rm d}\xi } ,\\
	\omega  \cdot {\partial _x}b\left( x \right) &= \omega  - {\mathscr{H}}_y^\nu \left( {x,0} \right) - \left\langle{\mathscr{H}}_{yy}^\nu \left( {x,0} \right), {\alpha  + {a_x}\left( x \right)} \right\rangle.\notag
\end{align}
Here, $ a(x)=a_{\nu}(x) $ and $ b(x)=b_\nu (x) $ are 1-periodic in all variables, and $ \alpha =\alpha_\nu$ is an infinite-dimensional  constant vector. To shorten notation, we drop the subscript $ \nu $. It should be  emphasized that, with the nondegeneracy condition  in \eqref{FULLninininini}, the second equation  \eqref{FULLgfe2} (and also the equivalent third one) plays an essential role in preserving the prescribed frequency of the unperturbed Hamiltonian system. Without this, the frequency may drift in each step of the KAM iteration.

Let us define  the error $ {\varepsilon _\nu } > 0 $ in the $ \nu $-th step of the KAM iteration to be the smallest number such that
\begin{align}
	\label{FULLerr1}&\left\|{{\mathscr{H}}^\nu }\left( {x,0} \right) - \int_{{\mathbb{T}^\infty }} {{{\mathscr{H}}^\nu }\left( {\xi ,0} \right){\rm d}\xi } \right\|_{{{\sigma _\nu }}} \leqslant {\varepsilon _\nu },\\
	\label{FULLerr2}&\left\|{\mathscr{H}}_y^\nu \left( {x,0} \right) - \omega \right\|_{{{\sigma _\nu }}} \leqslant \exp \left( {  \frac{\tau }{{{{\left( {{\sigma _\nu } - {\sigma _{\nu  + 1}}} \right)}^{\frac{1}{\eta} }}}}\log \left( {\frac{\tau }{{{\sigma _\nu } - {\sigma _{\nu  + 1}}}}} \right)} \right){\varepsilon _\nu },
\end{align}
and assume that
\begin{equation}\notag
	\left\|{\mathscr{H}}_{yy}^\nu \left( {x,y} \right)\right\|_{{{\sigma _\nu },{\sigma _\nu }}} \leqslant {M_\nu } \leqslant M,\quad \left\|{\left( {\int_{{\mathbb{T}^\infty }} {{{\mathscr{H}}_{yy}}\left( {x,0} \right){\rm d}x} } \right)^{ - 1}}\right\|_{{\mathscr{I}_\sigma^\infty}} \leqslant {M_\nu } \leqslant M
\end{equation}
for a universal constant $ M>0 $. It will be demonstrated that, by virtue of the smallness of the KAM error, the aforementioned reversibility can be preserved throughout the KAM process. 

 Define $ {\rho _j}: = \left( {\left( {8 - j} \right){\sigma _\nu } + j{\sigma _{\nu  + 1}}} \right)/8 $ for $ 0 \leqslant j \leqslant 4 $. Then, it follows that $ {\rho _0} = {\sigma _\nu }$ and ${\rho _4} = \left( {{\sigma _\nu } + {\sigma _{\nu  + 1}}} \right)/2 $. We aim to  establish the following estimates for $ a(x) $ and $ b(x) $ derived from the generating function method.

\begin{lemma}There hold
\begin{align*}
	\left\|a\left( x \right)\right\|_{{{ \frac{ {{\sigma _\nu } + {\sigma _{\nu  + 1}}}}{2} }}} &\leqslant \exp \left( {\frac{\tau }{{{{\left( {{\left( {{\sigma _\nu } + {\sigma _{\nu  + 1}}} \right)/2  } - {\sigma _{\nu  + 1}}} \right)}^{\frac{1}{\eta} }}}}\log \left( {\frac{\tau }{{{\sigma _\nu } - {\sigma _{\nu  + 1}}}}} \right)} \right){\varepsilon _\nu },\\
	\left\|{a_x}\right\|_{{{\frac{ {{\sigma _\nu } + {\sigma _{\nu  + 1}}}}{2} }}} &\leqslant c\exp \left( {\frac{\tau }{{{{\left( {{\sigma _\nu } - {\sigma _{\nu  + 1}}} \right)}^{\frac{1}{\eta} }}}}\log \left( {\frac{\tau }{{{\sigma _\nu } - {\sigma _{\nu  + 1}}}}} \right)} \right){\varepsilon _\nu },\\
	\left\|\alpha  + {a_x}\left( x \right)\right\|_{{{\frac{ {{\sigma _\nu } + {\sigma _{\nu  + 1}}}}{2} }}} &\leqslant cM\exp \left( {\frac{\tau }{{{{\left( {{\sigma _\nu } - {\sigma _{\nu  + 1}}} \right)}^{\frac{1}{\eta} }}}}\log \left( {\frac{\tau }{{{\sigma _\nu } - {\sigma _{\nu  + 1}}}}} \right)} \right){\varepsilon _\nu },\\
	\left\|{b_x}\left( x \right)\right\|_{{{\frac{ {{\sigma _\nu } + {\sigma _{\nu  + 1}}}}{2} }}} &\leqslant cM\exp \left( {\frac{\tau }{{{{\left( {{\sigma _\nu } - {\sigma _{\nu  + 1}}} \right)}^{\frac{1}{\eta} }}}}\log \left( {\frac{\tau }{{{\sigma _\nu } - {\sigma _{\nu  + 1}}}}} \right)} \right){\varepsilon _\nu },
\end{align*}
where $ c = c\left( 1 \right) > 0 $ is the constant given in Lemma \ref{FULLcauchy}, and we may regard it as a universal constant.
\end{lemma}
\begin{proof}

By \eqref{FULLgfe1} and Lemma \ref{FULLtdfcyl}, we have
\begin{align}
	\left\|a\left( x \right)\right\|_{{{\rho _1}}} &\leqslant \exp \left( {\frac{\tau }{{{{\left( {{\rho _0} - {\rho _1}} \right)}^{\frac{1}{\eta} }}}}\log \left( {\frac{\tau }{{{\rho _0} - {\rho _1}}}} \right)} \right)\left\|{{\mathscr{H}}^\nu }\left( {x,0} \right) - \int_{{\mathbb{T}^\infty }} {{{\mathscr{H}}^\nu }\left( {\xi ,0} \right){\rm d}\xi } \right\|_{{{\rho _0}}}\notag \\
	&\leqslant \exp \left( {\frac{{{8^{\frac{1}{\eta} }}\tau }}{{{{\left( {{\sigma _\nu } - {\sigma _{\nu  + 1}}} \right)}^{\frac{1}{\eta} }}}}\log \left( {\frac{{8\tau }}{{{\sigma _\nu } - {\sigma _{\nu  + 1}}}}} \right)} \right){\varepsilon _\nu }\notag \\
	\label{FULLa}& \leqslant \exp \left( {\frac{\tau }{{{{\left( {{\sigma _\nu } - {\sigma _{\nu  + 1}}} \right)}^{\frac{1}{\eta} }}}}\log \left( {\frac{\tau }{{{\sigma _\nu } - {\sigma _{\nu  + 1}}}}} \right)} \right){\varepsilon _\nu },
\end{align}
provided that $ \tau  = \tau \left( {\eta ,\mu } \right) > 0 $ is a universal constant. Therefore, with Lemma \ref{FULLcauchy} and \eqref{FULLa}, we obtain
\begin{align}
	\left\|{a_x}\right\|_{{{\rho _2}}} &\leqslant \frac{c}{{{\rho _1} - {\rho _2}}}\left\|a\right\|_{{{\rho _1}}}\notag \\
	& \leqslant \frac{8c}{{{\sigma _\nu } - {\sigma _{\nu  + 1}}}}\exp \left( {\frac{\tau }{{{{\left( {{\sigma _\nu } - {\sigma _{\nu  + 1}}} \right)}^{\frac{1}{\eta} }}}}\log \left( {\frac{\tau }{{{\sigma _\nu } - {\sigma _{\nu  + 1}}}}} \right)} \right){\varepsilon _\nu }\notag \\
	& = c\exp \left( {\frac{\tau }{{{{\left( {{\sigma _\nu } - {\sigma _{\nu  + 1}}} \right)}^{\frac{1}{\eta} }}}}\log \left( {\frac{\tau }{{{\sigma _\nu } - {\sigma _{\nu  + 1}}}}} \right) + \log \left( {\frac{8}{{{\sigma _\nu } - {\sigma _{\nu  + 1}}}}} \right)} \right){\varepsilon _\nu }\notag \\
	\label{FULLAX1}& \leqslant c\exp \left( {\frac{\tau }{{{{\left( {{\sigma _\nu } - {\sigma _{\nu  + 1}}} \right)}^{\frac{1}{\eta} }}}}\log \left( {\frac{\tau }{{{\sigma _\nu } - {\sigma _{\nu  + 1}}}}} \right)} \right){\varepsilon _\nu }.
\end{align}
By \eqref{FULLgfe2}, we have the following relation
\[\left\langle\int_{{\mathbb{T}^\infty }} {{\mathscr{H}}_{yy}^\nu \left( {\xi ,0} \right) {\rm d}\xi },\alpha\right\rangle  = \int_{{\mathbb{T}^\infty }} {\left( {\omega  - {\mathscr{H}}_y^\nu \left( {\xi ,0} \right)} \right){\rm d}\xi }  + \int_{{\mathbb{T}^\infty }} {\left\langle{\mathscr{H}}_{yy}^\nu \left( {\xi ,0} \right),{a_x}\left( \xi  \right)\right\rangle{\rm d}\xi } .\]
Then, we arrive at
\begin{align}
	\left\|\alpha \right\|_{{{\rho _2}}} &\leqslant M\left( {\left\|\omega  - {\mathscr{H}}_y^\nu \left( {x ,0} \right)\right\|_{{{\rho _2}}} + M\left\|{a_x}\left( x \right)\right\|_{{{\rho _2}}}} \right)\notag \\
	\label{FULLalpha}& \leqslant cM\exp \left( {\frac{\tau }{{{{\left( {{\sigma _\nu } - {\sigma _{\nu  + 1}}} \right)}^{\frac{1}{\eta}}}}}\log \left( {\frac{\tau }{{{\sigma _\nu } - {\sigma _{\nu  + 1}}}}} \right)} \right){\varepsilon _\nu },
\end{align}
and therefore
\begin{equation}\label{FULLalpha+ax}
	\left\|\alpha  + {a_x}\left( x \right)\right\|_{{{\rho _2}}} \leqslant \left\|\alpha \right\|_{{{\rho _2}}} + \left\|{a_x}\left( x \right)\right\|_{{{\rho _2}}} \leqslant cM\exp \left( {\frac{\tau }{{{{\left( {{\sigma _\nu } - {\sigma _{\nu  + 1}}} \right)}^{\frac{1}{\eta} }}}}\log \left( {\frac{\tau }{{{\sigma _\nu } - {\sigma _{\nu  + 1}}}}} \right)} \right){\varepsilon _\nu }.
\end{equation}
Further, combining \eqref{FULLalpha}, \eqref{FULLalpha+ax} and Lemma \ref{FULLtdfcyl}, we have
\begin{align}
	\left\|b\left( x \right)\right\|_{{{\rho _3}}} &\leqslant \exp \left( {\frac{\tau }{{{{\left( {{\rho _2} - {\rho _3}} \right)}^{\frac{1}{\eta} }}}}\log \left( {\frac{\tau }{{{\rho _2} - {\rho _3}}}} \right)} \right)\left\|\omega  - {\mathscr{H}}_y^\nu \left( {x,0} \right) - \left\langle{\mathscr{H}}_{yy}^\nu \left( {x,0} \right), {\alpha  + {a_x}\left( x \right)} \right\rangle\right\|_{{{\rho _2}}}\notag \\
	& \leqslant \exp \left( {\frac{\tau }{{{{\left( {{\sigma _\nu } - {\sigma _{\nu  + 1}}} \right)}^{\frac{1}{\eta} }}}}\log \left( {\frac{\tau }{{{\sigma _\nu } - {\sigma _{\nu  + 1}}}}} \right)} \right)\left( {\left\|\omega  - {\mathscr{H}}_y^\nu \left( {x,0} \right)\right\|_{{{\rho _2}}} + \left\|\left\langle{\mathscr{H}}_{yy}^\nu \left( {x,0} \right), {\alpha  + {a_x}\left( x \right)} \right\rangle\right\|_{{{\rho _2}}}} \right)\notag \\
	&\leqslant \exp \left( {\frac{\tau }{{{{\left( {{\sigma _\nu } - {\sigma _{\nu  + 1}}} \right)}^{\frac{1}{\eta} }}}}\log \left( {\frac{\tau }{{{\sigma _\nu } - {\sigma _{\nu  + 1}}}}} \right)} \right) \cdot cM\exp \left( {\frac{\tau }{{{{\left( {{\sigma _\nu } - {\sigma _{\nu  + 1}}} \right)}^{\frac{1}{\eta} }}}}\log \left( {\frac{\tau }{{{\sigma _\nu } - {\sigma _{\nu  + 1}}}}} \right)} \right){\varepsilon _\nu }\notag \\
	\label{FULLb}& \leqslant cM\exp \left( {\frac{\tau }{{{{\left( {{\sigma _\nu } - {\sigma _{\nu  + 1}}} \right)}^{\frac{1}{\eta} }}}}\log \left( {\frac{\tau }{{{\sigma _\nu } - {\sigma _{\nu  + 1}}}}} \right)} \right){\varepsilon _\nu }.
\end{align}
Finally, by \eqref{FULLb} and Lemma \ref{FULLcauchy}, we have
\begin{align}
	\left\|{b_x}\left( x \right)\right\|_{{{\rho _4}}} &\leqslant \frac{c}{ {{\rho _3} - {\rho _4}} }\left\|b\left( x \right)\right\|_{{{\rho _3}}}\notag \\
	& \leqslant \frac{{8c}}{{{\sigma _\nu } - {\sigma _{\nu  + 1}}}} \cdot cM\exp \left( {\frac{\tau }{{{{\left( {{\sigma _\nu } - {\sigma _{\nu  + 1}}} \right)}^{\frac{1}{\eta} }}}}\log \left( {\frac{\tau }{{{\sigma _\nu } - {\sigma _{\nu  + 1}}}}} \right)} \right){\varepsilon _\nu }\notag \\
	\label{FULLbx}& \leqslant cM\exp \left( {\frac{\tau }{{{{\left( {{\sigma _\nu } - {\sigma _{\nu  + 1}}} \right)}^{\frac{1}{\eta} }}}}\log \left( {\frac{\tau }{{{\sigma _\nu } - {\sigma _{\nu  + 1}}}}} \right)} \right){\varepsilon _\nu }.
\end{align}
This proves the lemma.
\end{proof}

 We have constructed functions  by taking advantage of generating functions in the form
\[{{\mathscr{U}}^\nu }\left( x \right) = \left\langle {\alpha ,x} \right\rangle  + a\left( x \right),\quad {{\mathscr{V}}^\nu }\left( x \right) = x + b\left( x \right).\]
Next, we define the symplectic and analytic transformation as 
\[z = {\psi ^\nu }\left( \zeta  \right),\quad z = \left( {x,y} \right),\quad \zeta  = \left( {\xi ,\kappa } \right) \Leftrightarrow \xi  = x + b\left( x \right),\quad y = \alpha  + {a_x}\left( x \right) + \kappa  + \left\langle{b_x}{\left( x \right)},\kappa\right\rangle .\]

\begin{lemma}\label{LM37}
The mapping $ z = {\psi ^\nu }\left( \zeta  \right) $ is well-defined, maps $ \left( {\xi ,\kappa } \right) \in {\mathscr{D}_{{\sigma _{\nu  + 1}},{\sigma _{\nu  + 1}}}} $ into $ \left( {x,y} \right) \in {\mathscr{D}_{\left( {{\sigma _{\nu  + 1}} + {\sigma _\nu }} \right)/2,\left( {{\sigma _{\nu  + 1}} + {\sigma _\nu }} \right)/2}} $, 
and satisfies the following estimates for $ \left( {\xi ,\kappa } \right) \in {\mathscr{D}_{{\sigma _{\nu  + 1}},{\sigma _{\nu  + 1}}}} $:
\begin{align}
	\left\|{\psi ^\nu }\left( \zeta  \right) - {\rm id} \right\|_{{{\sigma _{\nu  + 1}},{\sigma _{\nu  + 1}}}} &\leqslant {{\rm e}^{ - {2^\nu }K{{\sigma ^{ - \frac{2}{\eta} }}}}},\notag \\
	\left\|\psi _\zeta ^\nu \left( \zeta  \right) - {\rm Id}\right\|_{{{\sigma _{\nu  + 1}},{\sigma _{\nu  + 1}}}} &\leqslant {{\rm e}^{ - {2^\nu }K{{\sigma ^{ - \frac{2}{\eta} }}}}},\notag
\end{align}
where $ K>0 $ is a universal constant independent of $ \sigma>0 $.
\end{lemma}
\begin{proof}
	Postponing the proof of the estimate for $\varepsilon_\nu$ to Lemma \ref{LM39}\footnote{This does not affect the essence of the matter. However, it allows the reader to more clearly see the nature of our KAM technique.},  we obtain that
\begin{align}
	\left\|x - \xi \right\|_{{{\sigma _\nu }}} &= \left\|b\left( x \right)\right\|_{{{\sigma _\nu }}} \leqslant cM\exp \left( {\frac{\tau }{{{{\left( {{\sigma _\nu } - {\sigma _{\nu  + 1}}} \right)}^{\frac{1}{\eta} }}}}\log \left( {\frac{\tau }{{{\sigma _\nu } - {\sigma _{\nu  + 1}}}}} \right)} \right){\varepsilon _\nu }\notag \\
	&\leqslant cM\exp \left( {\frac{{{2^{\frac{1}{\eta} }}\tau }}{{{{\left( \sigma{\left( {1 - q} \right){q^\nu }} \right)}^{\frac{1}{\eta} }}}}\log \left( {\frac{{2\tau }}{{ \sigma \left( {1 - q} \right){q^\nu }}}} \right)} \right) \cdot {{\rm e}^{ - {2^\nu }K{{\sigma ^{ - \frac{2}{\eta} }}}}}\notag \\
	\label{FULLx-xi}& \leqslant {{\rm e}^{ - {2^\nu }K{{\sigma ^{ - \frac{2}{\eta} }}} + {\varrho ^\nu }K{{\sigma ^{ - \frac{2}{\eta} }}}}}  \leqslant {\rm e}^{ - {2^\nu }K{{\sigma ^{ - \frac{2}{\eta} }}}} \leqslant \frac{{{\sigma _\nu } - {\sigma _{\nu  + 1}}}}{8},
\end{align}
where $ 0<\varrho<2 $ is an appropriate constant due to our choice of $ q $, i.e., $ q>2^{-1/\eta} $, and
\begin{align*}
	\left\|{b_x}\left( x \right)\right\|_{{{\sigma _\nu }}} &\leqslant cM\exp \left( {\frac{\tau }{{{{\left( {{\sigma _\nu } - {\sigma _{\nu  + 1}}} \right)}^{\frac{1}{\eta} }}}}\log \left( {\frac{\tau }{{{\sigma _\nu } - {\sigma _{\nu  + 1}}}}} \right)} \right){\varepsilon _\nu } \\
	&\leqslant {{\rm e}^{ - {2^\nu }K{{\sigma ^{ - \frac{2}{\eta} }}}}} \leqslant \frac{{{\sigma _\nu } - {\sigma _{\nu  + 1}}}}{8}.
\end{align*}
Therefore, we get
\[\left\|x - \xi \right\|_{{\frac{{{\sigma _{\nu  + 1}} + {\sigma _\nu }} }{2}}},\left\|{b_x}\left( x \right)\right\|_{{\frac{{{\sigma _{\nu  + 1}} + {\sigma _\nu }} }{2}}} \leqslant {{\rm e}^{ - {2^\nu }K{{\sigma ^{ - \frac{2}{\eta} }}}}} \leqslant \frac{{{\sigma _\nu } - {\sigma _{\nu  + 1}}}}{8}.\]
Now, let 
$ \left( {\xi ,\kappa } \right) \in {\mathscr{D}_{\left( {3{\sigma _{\nu  + 1}} + {\sigma _\nu }} \right)/4,\left( {3{\sigma _{\nu  + 1}} + {\sigma _\nu }} \right)/4}} $, and let $ x \in \mathbb{T}_{\left( {{\sigma _{\nu  + 1}} + {\sigma _\nu }} \right)/2}^\infty  $ be the unique vector such that $ x + b\left( x \right) = \xi  $, and define
\begin{equation}\label{INFIY}
	y: = \alpha  + {a_x}\left( x \right) + \kappa  + \left\langle{b_x}{\left( x \right)},\kappa\right\rangle .
\end{equation}
Using infinite-dimensional Fourier analysis, we can prove that $ {\left\| y \right\|^ *_\varsigma } = \sum\nolimits_{j \in \mathbb{Z}} {\left| {{y_j}} \right|{{\left\langle j \right\rangle }^{\varsigma}}}   $ is small with respect to $ \nu $. Recalling \eqref{INFIY}, it suffices to show the smallness of $ a $ (namely the smallness of $ a_j $ with respect to $ j $ and $ \nu $). Note that 
\[a\left( x \right) = {a_\nu }\left( x \right) = {\left( {\omega  \cdot {\partial _x}} \right)^{ - 1}}\left( {\int_{{\mathbb{T}^\infty }} {{\mathscr{H}^\nu }\left( {\xi ,0} \right){\rm{d}}\xi }  - {\mathscr{H}^\nu }\left( {x,0} \right)} \right).\]
Here, $ {\left( {\omega  \cdot {\partial _x}} \right)^{ - 1}} $ is the familiar linear differential operator giving rise to small divisors. Then 
\[{\left\| a \right\|_{{\sigma _\nu }}} = \sum\limits_{0 \ne k \in \mathbb{Z}_ * ^\infty } {\frac{{\left| {\widehat {\mathscr{H}}_k^\nu \left( 0 \right)} \right|}}{{\left| {\left\langle {k,\omega } \right\rangle } \right|}}{{\rm{e}}^{{\sigma _\nu }{{\left| k \right|}_\eta }}}} \]
admits the estimate in \eqref{FULLa}. Similarly, $ a_x $ admits a Fourier expansion form with the estimate in \eqref{FULLAX1}. Below, let us assume $ \eta \geqslant \varsigma+2 $ without loss of generality. Now, with $ {\left| k \right|_\eta } = \sum\nolimits_{\ell  \in \mathbb{Z}} {{{\left\langle \ell  \right\rangle }^\eta }\left| {{k_\ell }} \right|}  \geqslant {\left\langle j \right\rangle ^\eta }\left| {{k_j}} \right| $ for all $ j \in \mathbb{Z} $, we obtain for $ x \in \mathbb{T}_{{\sigma _\nu }}^\infty  $ that
\begin{align}
	\left| {{{\left( {{a_x}} \right)}_j}} \right| &\leqslant \sum\limits_{0 \ne k \in \mathbb{Z}_ * ^\infty } {\frac{{\left| {\widehat {\mathscr{H}}_k^\nu \left( 0 \right)} \right|}}{{\left| {\left\langle {k,\omega } \right\rangle } \right|}}\left| {{k_j}} \right|\left| {{\mathrm{e}^{\mathrm{i}\left\langle {k,x} \right\rangle }}} \right|} \notag \\
	& \leqslant \frac{1}{{{{\left\langle j \right\rangle }^\eta }}}\sum\limits_{0 \ne k \in \mathbb{Z}_ * ^\infty } {\frac{{\left| {\widehat {\mathscr{H}}_k^\nu \left( 0 \right)} \right|}}{{\left| {\left\langle {k,\omega } \right\rangle } \right|}}{{\left\langle j \right\rangle }^\eta }\left| {{k_j}} \right|\left| {\exp \left( {\sum\limits_{s \in \mathbb{Z}} {{|k_s|}\left| {\operatorname{Im} {x_s}} \right|} } \right)} \right|}\notag \\
	& \leqslant \frac{1}{{{{\left\langle j \right\rangle }^\eta }}}\sum\limits_{0 \ne k \in \mathbb{Z}_ * ^\infty } {\frac{{\left| {\widehat {\mathscr{H}}_k^\nu \left( 0 \right)} \right|}}{{\left| {\left\langle {k,\omega } \right\rangle } \right|}}{{\left| k \right|}_\eta }\left| {\exp \left( {\sigma \sum\limits_{s \in \mathbb{Z}} {{{\left\langle s \right\rangle }^\eta }{|k_s|}} } \right)} \right|} \notag\\
	& = \frac{1}{{{{\left\langle j \right\rangle }^\eta }}}\sum\limits_{0 \ne k \in \mathbb{Z}_ * ^\infty } {\frac{{\left| {\widehat {\mathscr{H}}_k^\nu \left( 0 \right)} \right|}}{{\left| {\left\langle {k,\omega } \right\rangle } \right|}}{{\left| k \right|}_\eta }{{\rm{e}}^{\sigma {{\left| k \right|}_\eta }}}} \notag\\
	\label{FULLGAI1}& \leqslant \frac{1}{{{{\left\langle j \right\rangle }^\eta }}}\exp \left( {\frac{\tau }{{{{\left( {{\sigma _\nu } - {\sigma _{\nu  + 1}}} \right)}^{\frac{1}{\eta} }}}}\log \left( {\frac{\tau }{{{\sigma _\nu } - {\sigma _{\nu  + 1}}}}} \right)} \right){\varepsilon _\nu },
\end{align}
and this leads to 
\begin{align}
	\sum\limits_{j \in \mathbb{Z}} {\left| {{{\left( {{a_x}} \right)}_j}} \right|{{\left\langle j \right\rangle }^\varsigma}} & \leqslant \exp \left( {\frac{\tau }{{{{\left( {{\sigma _\nu } - {\sigma _{\nu  + 1}}} \right)}^{\frac{1}{\eta} }}}}\log \left( {\frac{\tau }{{{\sigma _\nu } - {\sigma _{\nu  + 1}}}}} \right)} \right){\varepsilon _\nu }\sum\limits_{j \in \mathbb{Z}} {\frac{1}{{{{\left\langle j \right\rangle }^{\eta-\varsigma  }}}}}\notag \\
	& \leqslant \exp \left( {\frac{\tau }{{{{\left( {{\sigma _\nu } - {\sigma _{\nu  + 1}}} \right)}^{\frac{1}{\eta} }}}}\log \left( {\frac{\tau }{{{\sigma _\nu } - {\sigma _{\nu  + 1}}}}} \right)} \right) \cdot \left( {{\varepsilon _\nu }\sum\limits_{j \in \mathbb{Z}} {\frac{1}{{{{\left\langle j \right\rangle }^2}}}} } \right)\notag\\
	\label{FULLGAI2}	& \leqslant \exp \left( {\frac{\tau }{{{{\left( {{\sigma _\nu } - {\sigma _{\nu  + 1}}} \right)}^{\frac{1}{\eta} }}}}\log \left( {\frac{\tau }{{{\sigma _\nu } - {\sigma _{\nu  + 1}}}}} \right)} \right),
\end{align}
as promised. Regarding the general case $ \varsigma>0 $, let us choose $ w\in \mathbb{N}^+ $ sufficiently large such that $ w \eta \geqslant \varsigma+2 $. With the Cauchy's estimate in the KAM process,  we could prove that $ D^m a $ also admits the same estimate in \eqref{FULLAX1}, or a slightly stronger version
\[\sum\limits_{0 \ne k \in \mathbb{Z}_ * ^\infty } {\frac{{\left| {\widehat {\mathscr{H}}_k^\nu \left( 0 \right)} \right|}}{{\left| {\left\langle {k,\omega } \right\rangle } \right|}}\left| k \right|_\eta ^m{{\mathrm{e}}^{\sigma {{\left| k \right|}_\eta }}}}  \leqslant \exp \left( {\frac{\tau }{{{{\left( {{\sigma _\nu } - {\sigma _{\nu  + 1}}} \right)}^{\frac{1}{\eta}}}}}\log \left( {\frac{\tau }{{{\sigma _\nu } - {\sigma _{\nu  + 1}}}}} \right)} \right){\varepsilon _\nu }.\]
This allows us to   adjust the leading coefficient in \eqref{FULLGAI1} to be $ {\left\langle j \right\rangle ^{ - w\eta }} $, in a similar way. Then the smallness in \eqref{FULLGAI2} is ensured by $ \sum\nolimits_{j \in \mathbb{Z}} {{{\left\langle j \right\rangle }^{ - \left( {w\eta  - \varsigma} \right)}}}  \leqslant \sum\nolimits_{j \in \mathbb{Z}} {{{\left\langle j \right\rangle }^{ - 2}}}  <  + \infty  $. It should be pointed out  that the analysis of $ a_j $ (or further $ y_j $) in this context is independent of the arithmetic  property of the prescribed frequency $ \omega $, therefore it remains valid in Theorem \ref{FULLT2}. We also note that the smallness estimate for $a_j$ could potentially be improved;  whenever $ k_j $ appears, we have $ {\left| k \right|_\eta } = \sum\nolimits_{\ell  \in \mathbb{Z}} {{{\left\langle \ell  \right\rangle }^\eta }\left| {{k_\ell }} \right|}  \geqslant {\left\langle j \right\rangle ^\eta }\left| {{k_j}} \right| \geqslant {\left\langle j \right\rangle ^\eta } $, indicating that we only require the estimate for the tail of the Fourier expansion (namely $ \sum\nolimits_{{{\left| k \right|}_\eta } \geqslant {{\left\langle j \right\rangle }^\eta }}  \cdots   $). However, we do not further investigate this possibility.

Moreover, by \eqref{FULLalpha+ax} and \eqref{FULLbx}, we have
\begin{align}
	\left\|y - \kappa \right\|_{{{\sigma _\nu }}} & \leqslant \left\|\alpha  + {a_x}\left( x \right)\right\|_{{{\sigma _\nu }}} + \left\|\left\langle{b_x}{\left( x \right)},\kappa\right\rangle \right\|_{{{\sigma _\nu }}}\notag \\
	&  \leqslant \left\|\alpha  + {a_x}\left( x \right)\right\|_{{{\sigma _\nu }}} + \left\|{b_x}\left( x \right)\right\|_{{{\sigma _\nu }}} \cdot \left|\kappa \right|\notag \\
	&  \leqslant cM\exp \left( {\frac{\tau }{{{{\left( {{\sigma _\nu } - {\sigma _{\nu  + 1}}} \right)}^{\frac{1}{\eta} }}}}\log \left( {\frac{\tau }{{{\sigma _\nu } - {\sigma _{\nu  + 1}}}}} \right)} \right){\varepsilon _\nu }\notag \\
	& \quad   + cM\exp \left( {\frac{\tau }{{{{\left( {{\sigma _\nu } - {\sigma _{\nu  + 1}}} \right)}^{\frac{1}{\eta}}}}}\log \left( {\frac{\tau }{{{\sigma _\nu } - {\sigma _{\nu  + 1}}}}} \right)} \right){\varepsilon _\nu } \cdot \frac{{{\sigma _\nu } + 3{\sigma _{\nu  + 1}}}}{4}\notag \\
	& \leqslant cM\exp \left( {\frac{\tau }{{{{\left( {{\sigma _\nu } - {\sigma _{\nu  + 1}}} \right)}^{\frac{1}{\eta}}}}}\log \left( {\frac{\tau }{{{\sigma _\nu } - {\sigma _{\nu  + 1}}}}} \right)} \right){\varepsilon _\nu }\notag \\
	\label{FULLy-eta}& \leqslant {{\rm e}^{ - {2^\nu }K{{\sigma ^{ - \frac{2}{\eta} }}}}}.
\end{align}
Then, from  \eqref{FULLx-xi} and \eqref{FULLy-eta}, we prove that
\begin{align*}
	\left\|{\psi ^\nu }\left( \zeta  \right) - {\rm id} \right\|_{{{\sigma _{\nu  + 1}},{\sigma _{\nu  + 1}}}} &\leqslant \left\|{\psi ^\nu }\left( \zeta  \right) - {\rm id} \right\|_{{\frac{{3{\sigma _{\nu  + 1}} + {\sigma _\nu }}}{4},\frac{{3{\sigma _{\nu  + 1}} + {\sigma _\nu }}}{4}}} \\
	&\leqslant \left\|x - \xi \right\|_{{{\sigma _\nu }}} + \left\|y - \kappa \right\|_{{{\sigma _\nu }}}\\ &\leqslant {{\rm e}^{ - {2^\nu }K{{\sigma ^{ - \frac{2}{\eta} }}}}}.
\end{align*}
Finally, using similar analysis of \eqref{FULLx-xi} and applying Lemma \ref{FULLcauchy},  we have
\begin{align*}
	\left\|\psi _\zeta ^\nu \left( \zeta  \right) - {\rm Id}\right\|_{{{\sigma _{\nu  + 1}},{\sigma _{\nu  + 1}}}} &\leqslant \frac{{8c}}{{{\sigma _\nu } - {\sigma _{\nu  + 1}}}}\left\|{\psi ^\nu }\left( \zeta  \right) - {\rm id} \right\|_{{\frac{{3{\sigma _{\nu  + 1}} + {\sigma _\nu }}}{4},\frac{{3{\sigma _{\nu  + 1}} + {\sigma _\nu }}}{4}}}\notag \\
	& \leqslant \exp \left( {\frac{\tau }{{{{\left( {{\sigma _\nu } - {\sigma _{\nu  + 1}}} \right)}^{\frac{1}{\eta} }}}}\log \left( {\frac{\tau }{{{\sigma _\nu } - {\sigma _{\nu  + 1}}}}} \right)} \right){\varepsilon _\nu } \cdot {{\rm e}^{ - {2^\nu }K{{\sigma ^{ - \frac{2}{\eta} }}}}}\notag \\ &\leqslant {{\rm e}^{ - {2^\nu }K{{\sigma ^{ - \frac{2}{\eta} }}}}}.
\end{align*}
This proves the lemma.
\end{proof}

 In view of Lemma \ref{LM37}, let us define the transformed Hamiltonian function $ {{\mathscr{H}}^{\nu  + 1}}: = {{\mathscr{H}}^\nu } \circ {\psi ^\nu } $ in the $ (\nu+1) $-th step. Then, we aim to establish the inductive step for $\varepsilon_\nu$.

\begin{lemma}\label{LEMMA38}
 The following inequalities are satisfied with $ \nu $ replaced by $ \nu + 1 $:
\begin{align*}
	&\left\|{{\mathscr{H}}^\nu }\left( {x,0} \right) - \int_{{\mathbb{T}^\infty }} {{{\mathscr{H}}^\nu }\left( {\xi ,0} \right){\rm d}\xi } \right\|_{{{\sigma _\nu }}} \leqslant {\varepsilon _\nu },\\
	&\left\|{\mathscr{H}}_y^\nu \left( {x,0} \right) - \omega \right\|_{{{\sigma _\nu }}} \leqslant \exp \left( {  \frac{\tau }{{{{\left( {{\sigma _\nu } - {\sigma _{\nu  + 1}}} \right)}^{\frac{1}{\eta} }}}}\log \left( {\frac{\tau }{{{\sigma _\nu } - {\sigma _{\nu  + 1}}}}} \right)} \right){\varepsilon _\nu },
\end{align*}
and
\begin{equation}\notag 
	\left\|{\mathscr{H}}_{yy}^\nu \left( z \right) - {\mathscr{H}}_{yy}^\nu \left( \zeta  \right)\right\|_{{{\sigma _\nu },{\sigma _\nu }}} \leqslant \sigma^{-1}{{\rm e}^{ - {2^\nu }K{{\sigma ^{ - \frac{2}{\eta} }}}}},
\end{equation}
for some universal constant $ K>0 $ independent of $ \sigma>0 $, where $ {{\mathscr{Q}}^0}: = {\mathscr{Q}}$ and ${{\mathscr{Q}}^\nu } = {\mathscr{H}}_{yy}^{\nu  - 1} $.
\end{lemma}
\begin{proof}
Denote $ z: = \left( {x,\alpha  + {a_x}} \right): = {\psi ^\nu }\left( {\xi ,0} \right) $ with $ \xi  \in \mathbb{T}_{{\sigma _{\nu  + 1}}}^\infty  $. Then, it is evident  that $ \left( {x,y} \right) \in {\mathscr{D}_{\left( {{\sigma _{\nu  + 1}} + {\sigma _\nu }} \right)/2,\left( {{\sigma _{\nu  + 1}} + {\sigma _\nu }} \right)/2}} $. Firstly, by \eqref{FULLerr2}, \eqref{FULLalpha+ax} and Lemma \ref{FULLtaylor}, we have
\begin{align}
	&\left\|{{\mathscr{H}}^{\nu  + 1}}\left( {\xi ,0} \right) - \int_{{\mathbb{T}^\infty }} {{{\mathscr{H}}^{\nu  + 1}}\left( {\chi ,0} \right){\rm d}\chi } \right\|_{{{\sigma _{\nu  + 1}}}}\notag \\
	\leqslant \;&\left\|{{\mathscr{H}}^{\nu  + 1}}\left( {\xi ,0} \right) - \left( {\int_{{\mathbb{T}^\infty }} {{{\mathscr{H}}^\nu }\left( {\chi ,0} \right){\rm d}\chi }  + \left\langle {\omega ,\alpha } \right\rangle } \right)\right\|_{{{\sigma _\nu }}}\notag \\
	\;& + \left\|\int_{{\mathbb{T}^\infty }} {\left( {{{\mathscr{H}}^{\nu  + 1}}\left( {\zeta ,0} \right) - \left( {\int_{{\mathbb{T}^\infty }} {{{\mathscr{H}}^\nu }\left( {\chi ,0} \right){\rm d}\chi }  + \left\langle {\omega ,\alpha } \right\rangle } \right)} \right){\rm d}\zeta } \right\|_{{{\sigma _\nu }}}\notag \\
	\leqslant \;&2\left\|{{\mathscr{H}}^{\nu  + 1}}\left( {\xi ,0} \right) - \left( {\int_{{\mathbb{T}^\infty }} {{{\mathscr{H}}^\nu }\left( {\chi ,0} \right){\rm d}\chi }  + \left\langle {\omega ,\alpha } \right\rangle } \right)\right\|_{{{\sigma _\nu }}}\notag \\
	= \;&2 {\left\|{{\mathscr{H}}^{\nu  + 1}}\left( {x,\alpha  + {a_x}} \right) - {{\mathscr{H}}^\nu }\left( {x,0} \right) -   \left\langle {\omega ,{\partial _x}a\left( x \right) } \right\rangle - \left\langle {\omega ,\alpha } \right\rangle \right\|_{{{\sigma _\nu }}}} \notag \\
	=\; &2 {\left\|{{\mathscr{H}}^{\nu  + 1}}\left( {x,\alpha  + {a_x}} \right) - {{\mathscr{H}}^\nu }\left( {x,0} \right) - \left\langle {{\mathscr{H}}_y^\nu \left( {x,0} \right),\alpha  + {a_x}} \right\rangle  + \left\langle {{\mathscr{H}}_y^\nu \left( {x,0} \right) - \omega ,\alpha  + {a_x}} \right\rangle \right\|_{{{\sigma _\nu }}}} \notag \\
	\leqslant \;&2\left( {\left\|{{\mathscr{H}}^{\nu  + 1}}\left( {x,\alpha  + {a_x}} \right) - {{\mathscr{H}}^\nu }\left( {x,0} \right) - \left\langle {{\mathscr{H}}_y^\nu \left( {x,0} \right),\alpha  + {a_x}} \right\rangle \right\|_{{{\sigma _\nu }}} + \left\|\left\langle {{\mathscr{H}}_y^\nu \left( {x,0} \right) - \omega ,\alpha  + {a_x}} \right\rangle \right\|_{{{\sigma _\nu }}}} \right)\notag \\
	\leqslant \;&2\left( {M\left\|\alpha  + {a_x}\right\|_{{\sigma _\nu }}^2 + \left\|{\mathscr{H}}_y^\nu \left( {x,0} \right) - \omega \right\|_{{{\sigma _\nu }}} \cdot \left\|\alpha  + {a_x}\right\|_{{{\sigma _\nu }}}} \right)\notag \\
	\leqslant \;&2\left( \begin{gathered}
		{c^2}{M^2}\exp \left( {\frac{{2\tau }}{{{{\left( {{\sigma _\nu } - {\sigma _{\nu  + 1}}} \right)}^{\frac{1}{\eta}}}}}\log \left( {\frac{\tau }{{{\sigma _\nu } - {\sigma _{\nu  + 1}}}}} \right)} \right)\varepsilon _\nu ^2 \hfill \\
		+ \exp \left( {\frac{\tau }{{{{\left( {{\sigma _\nu } - {\sigma _{\nu  + 1}}} \right)}^{\frac{1}{\eta} }}}}\log \left( {\frac{\tau }{{{\sigma _\nu } - {\sigma _{\nu  + 1}}}}} \right)} \right){\varepsilon _\nu } \cdot cM\exp \left( {\frac{\tau }{{{{\left( {{\sigma _\nu } - {\sigma _{\nu  + 1}}} \right)}^{\frac{1}{\eta} }}}}\log \left( {\frac{\tau }{{{\sigma _\nu } - {\sigma _{\nu  + 1}}}}} \right)} \right){\varepsilon _\nu } \hfill \\
	\end{gathered}  \right)\notag \\
	\leqslant \;&\exp \left( {\frac{{{\tau ^ * }}}{{{{\left( {{\sigma _\nu } - {\sigma _{\nu  + 1}}} \right)}^{\frac{1}{\eta} }}}}\log \left( {\frac{{{\tau ^ * }}}{{{\sigma _\nu } - {\sigma _{\nu  + 1}}}}} \right)} \right)\varepsilon _\nu ^2,\notag
\end{align}
where $ {\tau ^ * } = {\tau ^ * }\left( {\tau ,c,M,\eta ,\mu } \right) > 0 $ is a universal constant independent of $ \sigma>0 $. Recalling \eqref{FULLerr1}, the above estimate implies that
\begin{equation}\label{FULLe1}
	{\varepsilon _{\nu  + 1}} \leqslant \exp \left( {\frac{{{\tau ^ * }}}{{{{\left( {{\sigma _\nu } - {\sigma _{\nu  + 1}}} \right)}^{\frac{1}{\eta} }}}}\log \left( {\frac{{{\tau ^ * }}}{{{\sigma _\nu } - {\sigma _{\nu  + 1}}}}} \right)} \right)\varepsilon _\nu ^2.
\end{equation}

Secondly, with \eqref{FULLerr2}, \eqref{FULLalpha+ax}, \eqref{FULLbx} and Lemma \ref{FULLtaylor}, we obtain that
\begin{align}
	&\left\|{\mathscr{H}}_y^{\nu  + 1}\left( {\xi ,0} \right) - \omega \right\|_{{{\sigma _{\nu  + 1}}}} \notag \\
	= \;&\left\|\left\langle {{\rm Id} + {b_x}}, {\mathscr{H}}_y^\nu \left( {x,\alpha  + {a_x}} \right)\right\rangle - \omega \right\|_{{{\sigma _{\nu  + 1}}}}\notag \\
	=\; &\left\|\left( \begin{gathered}
		{\mathscr{H}}_y^\nu \left( {x,\alpha  + {a_x}} \right) - {\mathscr{H}}_y^\nu \left( {x,0} \right) - \left\langle{\mathscr{H}}_{yy}^\nu \left( {x,0} \right), {\alpha  + {a_x}} \right\rangle \hfill \\
		+ \left\langle{b_x}, {{\mathscr{H}}_y^\nu \left( {x,\alpha  + {a_x}} \right) - {\mathscr{H}}_y^\nu \left( {x,0} \right)} \right\rangle \hfill \\
		+\left\langle {b_x}, {{\mathscr{H}}_y^\nu \left( {x,0} \right)} \right\rangle - \omega  \hfill \\
	\end{gathered}  \right)\right\|_{{{\sigma _{\nu  + 1}}}}\notag \\
	\leqslant\;& \left\|{\mathscr{H}}_y^\nu \left( {x,\alpha  + {a_x}} \right) - {\mathscr{H}}_y^\nu \left( {x,0} \right) -\left\langle {\mathscr{H}}_{yy}^\nu \left( {x,0} \right) ,{\alpha  + {a_x}} \right\rangle\right\|_{{{\sigma _\nu }}}\notag \\
\quad &+ \left\|\left\langle{b_x}, {{\mathscr{H}}_y^\nu \left( {x,\alpha  + {a_x}} \right) - {\mathscr{H}}_y^\nu \left( {x,0} \right)} \right\rangle\right\|_{{{\sigma _\nu }}}\notag \\
	\quad &+ \left\|\left\langle{b_x}, {{\mathscr{H}}_y^\nu \left( {x,0} \right)}  \right\rangle - \omega \right\|_{{{\sigma _\nu }}}\notag \\
	\leqslant \;&\frac{{M\left\|\alpha  + {a_x}\right\|_{{\sigma _\nu }}^2}}{{{\sigma _\nu } - \left\|\alpha  + {a_x}\right\|_{{{\sigma _\nu }}}}} + \left\|{b_x}\right\|_{{{\sigma _\nu }}} \cdot M\left\|\alpha  + {a_x}\right\|_{{{\sigma _\nu }}} + \left\|{b_x}\right\|_{{{\sigma _\nu }}} \cdot \left\|{\mathscr{H}}_y^\nu \left( {x,0} \right) - \omega \right\|_{{{\sigma _\nu }}}\notag \\
	\leqslant\;& \frac{{8{c^2}{M^3}}}{{{\sigma _\nu } - {\sigma _{\nu  + 1}}}}\exp \left( {\frac{{2\tau }}{{{{\left( {{\sigma _\nu } - {\sigma _{\nu  + 1}}} \right)}^{\frac{1}{\eta} }}}}\log \left( {\frac{\tau }{{{\sigma _\nu } - {\sigma _{\nu  + 1}}}}} \right)} \right)\varepsilon _\nu ^2\notag \\
\quad &+ cM\exp \left( {\frac{{2\tau }}{{{{\left( {{\sigma _\nu } - {\sigma _{\nu  + 1}}} \right)}^{\frac{1}{\eta} }}}}\log \left( {\frac{\tau }{{{\sigma _\nu } - {\sigma _{\nu  + 1}}}}} \right)} \right){\varepsilon _\nu } \cdot cM\exp \left( {\frac{\tau }{{{{\left( {{\sigma _\nu } - {\sigma _{\nu  + 1}}} \right)}^{\frac{1}{\eta} }}}}\log \left( {\frac{\tau }{{{\sigma _\nu } - {\sigma _{\nu  + 1}}}}} \right)} \right){\varepsilon _\nu }\notag \\
\quad  &+ cM\exp \left( {\frac{{2\tau }}{{{{\left( {{\sigma _\nu } - {\sigma _{\nu  + 1}}} \right)}^{\frac{1}{\eta} }}}}\log \left( {\frac{\tau }{{{\sigma _\nu } - {\sigma _{\nu  + 1}}}}} \right)} \right){\varepsilon _\nu } \cdot \exp \left( {\frac{\tau }{{{{\left( {{\sigma _\nu } - {\sigma _{\nu  + 1}}} \right)}^{\frac{1}{\eta} }}}}\log \left( {\frac{\tau }{{{\sigma _\nu } - {\sigma _{\nu  + 1}}}}} \right)} \right){\varepsilon _\nu }\notag \\
	\leqslant\; &\exp \left( {\frac{{{\tau ^ * }}}{{{{\left( {{\sigma _\nu } - {\sigma _{\nu  + 1}}} \right)}^{\frac{1}{\eta}}}}}\log \left( {\frac{{{\tau ^ * }}}{{{\sigma _\nu } - {\sigma _{\nu  + 1}}}}} \right)} \right)\varepsilon _\nu ^2.\notag
\end{align}
Recalling \eqref{FULLerr2},  we have
\[\exp \left( {\frac{{{\tau ^ * }}}{{{{\left( {{\sigma _{\nu  + 1}} - {\sigma _{\nu  + 2}}} \right)}^{\frac{1}{\eta} }}}}\log \left( {\frac{{{\tau ^ * }}}{{{\sigma _{\nu  + 1}} - {\sigma _{\nu  + 2}}}}} \right)} \right){\varepsilon _{\nu  + 1}} \leqslant \exp \left( {\frac{{{\tau ^ * }}}{{{{\left( {{\sigma _\nu } - {\sigma _{\nu  + 1}}} \right)}^{\frac{1}{\eta} }}}}\log \left( {\frac{{{\tau ^ * }}}{{{\sigma _\nu } - {\sigma _{\nu  + 1}}}}} \right)} \right)\varepsilon _\nu ^2,\]
which yields that
\begin{equation}\label{FULLe2}
	{\varepsilon _{\nu  + 1}} \leqslant \exp \left( {\frac{{{\tau ^ * }}}{{{{\left( {{\sigma _\nu } - {\sigma _{\nu  + 1}}} \right)}^{\frac{1}{\eta} }}}}\log \left( {\frac{{{\tau ^ * }}}{{{\sigma _\nu } - {\sigma _{\nu  + 1}}}}} \right)} \right)\varepsilon _\nu ^2,\quad  \nu \in \mathbb{N},
\end{equation}
i.e., the same as \eqref{FULLe1}. We assert that   the convergence rate of $ {\varepsilon _\nu } $ is super-exponential, and the detailed  proof will be given in Lemma \ref{LM39}.

Finally, we have
\begin{align}
	\left\|{\mathscr{H}}_{yy}^\nu \left( z \right) - {\mathscr{H}}_{yy}^\nu \left( \zeta  \right)\right\|_{{{\sigma _\nu },{\sigma _\nu }}} &= \left\|\frac{1}{{2\pi {\rm i}}}\int_{\widetilde \Gamma } {\frac{1}{{\lambda \left( {\lambda  - 1} \right)}}\mathscr{H}_{yy}^\nu \left( {\zeta  + \lambda \left( {z - \zeta } \right)} \right){\rm d}\lambda } \right\|_{{{\sigma _\nu },{\sigma _\nu }}}\notag \\
	&\leqslant \frac{M}{{2\pi }}\max \left\{ {\frac{{\left\|x - \xi \right\|_{{{\sigma _\nu }}}}}{{{\sigma _\nu } - \left\|\xi \right\|_{{{\sigma _\nu }}} - \left\|x - \xi \right\|_{{{\sigma _\nu }}}}},\frac{{\left\|y - \kappa \right\|_{{{\sigma _\nu }}}}}{{{\sigma _\nu } - \left\|\kappa \right\|_{{{\sigma _\nu }}} - \left\|y - \kappa \right\|_{{{\sigma _\nu }}}}}} \right\}\notag \\
	\label{FULLHessen1} &\leqslant \frac{{{\varepsilon _\nu }}}{{{\sigma _\nu } - {\sigma _{\nu  + 1}}}} \leqslant \frac{{2{{\rm e}^{ - {2^\nu }K{{\sigma ^{ - \frac{2}{\eta} }}}}}}}{\sigma{\left( {1 - q} \right){q^\nu }}} \leqslant \sigma^{-1}{{\rm e}^{ - {2^\nu }K{{\sigma ^{ - \frac{2}{\eta} }}}}},
\end{align}
in which the estimate for $ \varepsilon _\nu $ given in Lemma \ref{LM39} is used, and the curve $ \widetilde \Gamma  $ is defined as follows:
\[\widetilde \Gamma : = \left\{ {\lambda  \in \mathbb{C}:\quad \left| \lambda  \right| = \min \left\{ {\frac{{{\sigma _\nu } - \left\|\xi \right\|_{{{\sigma _\nu }}}}}{{\left\|x - \xi \right\|_{{{\sigma _\nu }}}}},\frac{{{\sigma _\nu } - \left\|\kappa \right\|_{{{\sigma _\nu }}}}}{{\left\|y - \kappa \right\|_{{{\sigma _\nu }}}}}} \right\} > 1} \right\}.\]
On the other hand, note that
\[{\mathscr{H}}_{yy}^{\nu  + 1}\left( {\xi ,\kappa } \right) = \left\langle  \left\langle{\mathscr{H}}_{yy}^\nu \left( {x,y} \right) , {{\rm Id} + b_x(x)} \right\rangle, {{\rm Id} + {b_x(x)}}\right\rangle.\]
Then, by \eqref{FULLbx} and the estimate in Lemma \ref{LM39}, we have
\begin{align}
\left\|{\mathscr{H}}_{yy}^{\nu  + 1}\left( \zeta  \right) - {\mathscr{H}}_{yy}^\nu \left( z \right)\right\|_{{{\sigma _{\nu  + 1}},{\sigma _{\nu  + 1}}}} &\leqslant 2\left\|{b_x}\left( x \right)\right\|_{{{\sigma _\nu }}} \cdot \left\|{\mathscr{H}}_{yy}^\nu \left( z \right)\right\|_{{{\sigma _\nu },{\sigma _\nu }}} + \left\|{b_x}\left( x \right)\right\|_{{\sigma _\nu }}^2 \cdot \left\|{\mathscr{H}}_{yy}^\nu \left( z \right)\right\|_{{{\sigma _\nu },{\sigma _\nu }}} \notag \\
	\label{Hessen2}&\leqslant {{\rm e}^{ - {2^\nu }K{{\sigma ^{ - \frac{2}{\eta} }}}}}.
\end{align}
By summing up \eqref{FULLHessen1} and \eqref{Hessen2} and comparing the order, we arrive at
\begin{equation}\notag
	\left\|{\mathscr{H}}_{yy}^\nu \left( {x,y} \right) - {{\mathscr{Q}}^\nu }\left( {x,y} \right)\right\|_{{{\sigma _{\nu  + 1}},{\sigma _{\nu  + 1}}}} \leqslant \sigma^{-1}{{\rm e}^{ - {2^\nu }K{{\sigma ^{ - \frac{2}{\eta} }}}}}.
\end{equation}

Recalling the estimates for $ {{\psi ^\nu }} $ and $ {{\mathscr{H}}^{\nu  + 1}} = {{\mathscr{H}}^\nu } \circ {\psi ^\nu } $ in Lemma \ref{LM37}, one can evidently verify that there exists a universal constant $ {M^ * } $ such that $ {M_\nu } \leqslant {M^ * } $, which completes the proof of the induction. Further,  in view of the definition of the KAM error $ {\varepsilon _\nu } $,  we only have to require the  initial conditions in \eqref{FULL1}, \eqref{FULL2} and \eqref{FULL3} to be satisfied:
\[\left\|{\mathscr{H}}\left( {x,0} \right) - \int_{{\mathbb{T}^\infty }} {{\mathscr{H}}\left( {\xi ,0} \right){\rm d}\xi } \right\|_{\sigma } \leqslant {\varepsilon _0} \leqslant {{\rm e}^{ - K{{\sigma ^{ - \frac{2}{\eta} }}}}},\]
and
\begin{align}
	\left\|{{\mathscr{H}}_y}\left( {x,0} \right) - \omega \right\|_{\sigma } & \leqslant \exp \left( {\frac{{{\tau ^ * }}}{{{{\left( {{\sigma _0} - {\sigma _1}} \right)}^{\frac{1}{\eta} }}}}\log \left( {\frac{{{\tau ^ * }}}{{{\sigma _0} - {\sigma _1}}}} \right)} \right){\varepsilon _0}\notag \\
	& = \exp \left( {\frac{{{\tau ^ * }}}{{{{\left( {{2^{ - 1}}\sigma \left( {1 - q } \right)} \right)}^{\frac{1}{\eta} }}}}\log \left( {\frac{{2{\tau ^ * }}}{{\sigma \left( {1 - q } \right)}}} \right)} \right) \cdot {{\rm e}^{ - K{{\sigma ^{ - \frac{2}{\eta}}}}}}\notag \\
	&\leqslant {{\rm e}^{ - K{{ {\sigma } }^{ - \frac{2}{\eta} }}}},\notag
\end{align}
and
\[\left\|{{\mathscr{H}}_{yy}}\left( {x,y} \right) - {\mathscr{Q}}\left( {x,y} \right)\right\|_{{\sigma ,\sigma }} \leqslant \sigma^{-1}{{\rm e}^{ - K{{\sigma ^{ - \frac{2}{\eta} }}}}}.\]
This proves the lemma.
\end{proof}

Note that for operators $ {\mathscr{B}_1},{\mathscr{B}_2} \in \mathscr{I}_\sigma ^\infty  $ with $ {\mathscr{B}_1} $ being invertible, we  have $ {\left( {{\mathscr{B}_1} + {\mathscr{B}_2}} \right)^{ - 1}} = \mathscr{B}_1^{ - 1}{\left( {{\rm Id} + \mathscr{B}_1^{ - 1}{\mathscr{B}_2}} \right)^{ - 1}} $. Then, by using the Neumann series argument, $ \mathscr{B}_1+\mathscr{B}_2 $ is also invertible, whenever the perturbation $ \mathscr{B}_2 $ is sufficiently small, in the sense that $ {\left\| {{\mathscr{B}_2}} \right\|_{\mathscr{I}_\sigma ^\infty }}\ll 1 $. It is evident   that $ {\left\| \mathscr{B} \right\|_{\sigma ,\sigma }} \ll 1$  is sufficient to ensure this point, because, by a similar argument to that in Lemma  \ref{FULLtaylor}, we have
\[{\left\| {{\mathscr{B}_2}} \right\|_{\mathscr{I}_\sigma ^\infty }} = {\sup _{{{\left\| u \right\|}_\sigma } = 1}}{\left\| {{\mathscr{B}_2}u} \right\|_\sigma } \leqslant {\sup _{{{\left\| u \right\|}_\sigma } = 1}}{\left\| \mathscr{B} \right\|_{\sigma ,\sigma }}{\left\| u \right\|_\sigma } = {\left\|\mathscr{B} \right\|_{\sigma ,\sigma }}.\]
As a consequence, according to Lemma \ref{LEMMA38}, it can be concluded that ${\left( {\int_{{\mathbb{T}^\infty }} {\mathscr{Q}\left( {x,0} \right){\rm d}x} } \right)^{ - 1}}$ remains invertible after perturbation. Therefore, the prescribed frequency $ \omega \in \mathbb{R}^{\mathbb{Z}} $ can be preserved during the KAM iteration. 

Now, we are in a position to prove that the KAM error $ \varepsilon _\nu $ is super-exponentially convergent, whenever $ \varepsilon_0>0 $ is sufficiently small.

\begin{lemma}\label{LM39}
The following estimate holds:
\begin{equation}\notag 
	\varepsilon _\nu  \leqslant {{\rm e}^{ - {2^\nu }K{{\sigma ^{ - \frac{2}{\eta} }}}}},\quad \nu  \in \mathbb{N},
\end{equation}
where  $ K>0 $ is a universal constant independent of $ \sigma>0 $.
\end{lemma}
\begin{proof}
Note that $ 1/2<{q^{1/\eta }}<1  $, since we have required that $ {2^{ - \eta }} < q < 2^{-1}\left({2^{ - \eta }}+1\right) $. Hence, there exists some $ \delta  = \delta \left( {\eta } \right) > 0 $ such that $ 1<d: = {\left( {{q^{1/\eta }} - \delta } \right)^{ - 1}} <2 $. Recalling \eqref{FULLe1} and \eqref{FULLe2}, we arrive at
\begin{align}
	{\varepsilon _{\nu  + 1}} &\leqslant \exp \left( {\frac{{{\tau ^ * }}}{{{{\left( {{\sigma _\nu } - {\sigma _{\nu  + 1}}} \right)}^{\frac{1}{\eta} }}}}\log \left( {\frac{{{\tau ^ * }}}{{{\sigma _\nu } - {\sigma _{\nu  + 1}}}}} \right)} \right)\varepsilon _\nu ^2\notag \\
	& \leqslant \exp \left( {\frac{{{\tau ^ * }}}{{{{\left( {{2^{ - 1}}\sigma \left( {1 - q} \right){q^\nu }} \right)}^{\frac{1}{\eta} }}}}\log \left( {\frac{{{\tau ^ * }}}{{{2^{ - 1}}\sigma \left( {1 - q} \right){q^\nu }}}} \right)} \right)\varepsilon _\nu ^2\notag \\
	& \leqslant \exp \left( {\frac{{{\tau ^ * }}}{{{{\left( {\sigma \left( {1 - q} \right){q^\nu }} \right)}^{\frac{1}{\eta} }}}}\left( {\log \left( {\frac{{{\tau ^ * }}}{{{q^\nu }}}} \right) + \log \left( {\frac{1}{{\sigma \left( {1 - q } \right)}}} \right)} \right)} \right)\varepsilon _\nu ^2\notag \\
	&= \exp \left( {\frac{{{\tau ^ * }}}{{{{\left( {{q^{\frac{1}{\eta} }}} \right)}^\nu }{{\left( {\sigma \left( {1 - q } \right)} \right)}^{\frac{1}{\eta} }}}}\left( {\log \left( {\frac{{{\tau ^ * }}}{{{q^\nu }}}} \right) + \eta \log \left( {\frac{1}{{{{\left( {\sigma \left( {1 - q } \right)} \right)}^{\frac{1}{\eta} }}}}} \right)} \right)} \right)\varepsilon _\nu ^2\notag \\
	& \leqslant \exp \left( {\frac{{{\tau ^ * }}}{{{{\left( {{q^{\frac{1}{\eta} }} - \delta } \right)}^\nu }{{\left( {\sigma \left( {1 - q } \right)} \right)}^{\frac{1}{\eta} }}}} \cdot \frac{\eta }{{{{\left( {\sigma \left( {1 - q } \right)} \right)}^{\frac{1}{\eta} }}}}} \right)\varepsilon _\nu ^2\notag \\
	&\leqslant \exp \left( {\frac{{{\tau ^ * }}}{{{{\left( {{q^{\frac{1}{\eta} }} - \delta } \right)}^\nu }{{ {\sigma } }^{\frac{2}{\eta} }}}}} \right)\varepsilon _\nu ^2\notag \\
	\label{FULLczs1}&  = \exp \left( {{\tau ^ * }{{ {\sigma } }^{ - \frac{2}{\eta} }}{d^\nu }} \right)\varepsilon _\nu ^2.
\end{align}
Therefore, with \eqref{FULLczs1}, we have
\[\log {\varepsilon _{\nu  + 1}} \leqslant {\tau ^ * }{\sigma ^{ - \frac{2}{\eta}}}{d^\nu } + 2\log {\varepsilon _\nu },\]
which is equivalent to (note that $ 1<d<2 $)
\begin{equation}\label{FULLLOGGG}
	\log {\varepsilon _{\nu  + 1}} + \frac{{{\tau ^ * }{\sigma ^{ - \frac{2}{\eta}}}}}{{2 - d}}{d^{\nu  + 1}} \leqslant 2\left( {\log {\varepsilon _\nu } + \frac{{{\tau ^ * }{\sigma ^{ - \frac{2}{\eta} }}}}{{2 - d}}{d^\nu }} \right).
\end{equation}
It is important to emphasize that we carefully select the value of $ q $ at the beginning of the proof to ensure the super-exponential property in this context. If $ 0<q\leqslant 2^{-\eta} $, then the previously used technique fails. Now, we derive from \eqref{FULLLOGGG} that
\[{\varepsilon _\nu } \leqslant \exp \left( {{2^\nu }\left( {\log {\varepsilon _0} + \frac{{{\tau ^ * }{\sigma ^{ - \frac{2}{\eta} }}}}{{2 - d}}} \right) - \frac{{{\tau ^ * }{\sigma ^{ - \frac{2}{\eta} }}}}{{2 - d}}{d^\nu }} \right),\]
and this implies the super-exponential property of $ \varepsilon _\nu $ as 
\[{\varepsilon _\nu } \leqslant {{\rm{e}}^{ - {2^\nu }K{\sigma ^{ - \frac{2}{\eta} }}}},\quad \nu  \in \mathbb{N},\]
where  $ K = K\left( {{\tau ^ * }} \right) = K\left( {\tau ,c,M,\eta ,\mu } \right) > 0 $ is  a sufficiently large constant  independent of $ \sigma>0 $, whenever we require the initial error $ \varepsilon _0 $ to be sufficiently small:
\[\log {\varepsilon _0} + \frac{{{\tau ^ * }{\sigma ^{ - \frac{2}{\eta} }}}}{{2 - d}} \ll -1.\]
This shows that the convergence rate of our  KAM iteration is indeed super-exponential.
\end{proof}

 It remains to establish the uniform convergence of the sequence
\[{\phi ^\nu }: = {\psi ^0} \circ {\psi ^1} \circ  \cdots  \circ {\psi ^\nu }\]
for $ \left( {\xi ,\kappa } \right) \in {\mathscr{D}_{\sigma /4,\sigma /4}} $, and also the estimates of the  transformed Hamiltonian function in Theorem \ref{FULLT1}. The analysis is straightforward, so we do not present it in the form of a lemma.

It can be obtained from Lemma \ref{LM37} that if $ \left( {\xi ,\kappa } \right) \in {\mathscr{D}_{{\sigma _\nu },{\sigma _\nu }}} $ and $ z: = {\psi ^{\ell  + 1}} \circ {\psi ^\ell } \circ  \cdots  \circ {\psi ^{\nu  - 1}}\left( \zeta  \right) $, then $ \left( {x,y} \right) \in {\mathscr{D}_{{\sigma _{\ell  + 1}},{\sigma _{\ell  + 1}}}} $, and  therefore
\begin{equation}\notag
	\left\|\psi _\zeta ^\ell \left( {{\psi ^{\ell  + 1}} \circ {\psi ^\ell } \circ  \cdots  \circ {\psi ^{\nu  - 1}}\left( \zeta  \right)} \right)\right\|_{{\frac{\sigma}{2},\frac{\sigma}{2}}} \leqslant 1 + {{\rm e}^{ - {2^\ell }K{{\sigma ^{ - \frac{2}{\eta} }}}}},
\end{equation}
which implies that (recall that $ 0<\sigma<1 $)
\begin{align*}
	\left\|\psi _\zeta ^{\nu  - 1}\left( \zeta  \right)\right\|_{{\frac{\sigma}{2},\frac{\sigma}{2}}} &\leqslant \prod\limits_{j = 0}^{\nu  - 1} {\left( {1 + {{\rm e}^{ - {2^j}K{{\sigma ^{ - \frac{2}{\eta} }}}}}} \right)}  = \exp \left( {\sum\limits_{j = 0}^{\nu  - 1} {\log \left( {1 + {{\rm e}^{ - {2^j}K{{\sigma ^{ - \frac{2}{\eta} }}}}}} \right)} } \right)\notag \\
	& \leqslant \exp \left( {\sum\limits_{j = 0}^\infty  {\frac{1}{{{{\rm e}^{{2^j}K{{\sigma ^{ - \frac{2}{\eta} }}}}}}}} } \right) \leqslant \exp \left( {\sum\limits_{j = 0}^\infty  {\frac{1}{{{2^j}K{{\sigma ^{ - \frac{2}{\eta} }}}}}} } \right) = \exp \left( {\frac{2{\sigma ^{\frac{2}{\eta} }}}{K}} \right)\leqslant 2.
\end{align*}
Then, it follows that
\begin{align*}
	\left\|{\phi ^\nu }\left( \zeta  \right) - {\phi ^{\nu  - 1}}\left( \zeta  \right)\right\|_{{\frac{\sigma}{2},\frac{\sigma}{2}}} &= \left\|\left( {{\psi ^0} \circ  \cdots  \circ {\psi ^{\nu  - 1}}} \right) \circ {\psi ^\nu }\left( \zeta  \right) - {\phi ^{\nu  - 1}}\left( \zeta  \right)\right\|_{{\frac{\sigma}{2},\frac{\sigma}{2}}} \notag \\
	&= \left\|{\phi ^{\nu  - 1}}\left( {{\psi ^\nu }\left( \zeta  \right)} \right) - {\phi ^{\nu  - 1}}\left( \zeta  \right)\right\|_{{\frac{\sigma}{2},\frac{\sigma}{2}}} \notag \\
	& \leqslant 2\left\|{\psi ^\nu }\left( \zeta  \right) - {\rm id} \right\|_{{\frac{\sigma}{2},\frac{\sigma}{2}}}\notag \\ 
	&\leqslant {{\rm e}^{ - {2^\nu }K{{\sigma ^{ - \frac{2}{\eta} }}}}}.
\end{align*}
This also holds for $ \nu = 0 $ if we define $ {\phi ^{ - 1}}: = \mathrm{id}$.

Now, the limit function $ \phi : = \mathop {\lim }\nolimits_{\nu  \to \infty } {\phi ^\nu } $ satisfies
\begin{align*}
\left\|\phi \left( \zeta  \right) - {\rm id} \right\|_{{\frac{\sigma}{2},\frac{\sigma}{2}}} \leqslant \sum\limits_{\nu  = 0}^\infty  {\left\|{\phi ^\nu }\left( \zeta  \right) - {\phi ^{\nu  - 1}}\left( \zeta  \right)\right\|_{{\frac{\sigma}{2},\frac{\sigma}{2}}}} \leqslant \sum\limits_{\nu  = 0}^\infty  {{{\rm e}^{ - {2^\nu }K{{\sigma ^{ - \frac{2}{\eta} }}}}}}  \leqslant \sum\limits_{\nu  = 0}^\infty  {\frac{1}{{{2^\nu }K{{\sigma ^{ - \frac{2}{\eta} }}}}}}  \leqslant \frac{2{\sigma ^{\frac{2}{\eta} }}}{K}.
\end{align*}
Then, by Lemma \ref{FULLcauchy}, we obtain that
\[\left\|{\phi _\zeta }\left( \zeta  \right) - {\rm Id}\right\|_{{\frac{\sigma}{4},\frac{\sigma}{4}}} \leqslant \frac{4}{\sigma }\left\|\phi \left( \zeta  \right) - {\rm id} \right\|_{{\frac{\sigma}{2},\frac{\sigma}{2}}} \leqslant \frac{8{\sigma ^{\frac{2}{\eta}-1 }}}{{ K}}.\]
This implies that $\phi$ is indeed an analytic and symplectic diffeomorphism.

Thus, the transformed Hamiltonian function can be written as
\[{\mathscr{W}}\left( \zeta  \right): = {\mathscr{H}} \circ \phi \left( \zeta  \right) = \mathop {\lim }\limits_{\nu  \to \infty } {\mathscr{H}} \circ {\phi ^0} \circ  \cdots  \circ {\phi ^\nu } = \mathop {\lim }\limits_{\nu  \to \infty } {{\mathscr{H}}^\nu }\]
for $ \left( {\xi ,\kappa } \right) \in {\mathscr{D}_{\sigma /4,\sigma /4}} $, which satisfies
\[{{\mathscr{W}}_\xi }\left( {\xi ,0} \right) = 0,\quad {{\mathscr{W}}_\kappa }\left( {\xi ,0} \right) = \omega ,\]
and
\begin{align}
	\left\|{{\mathscr{W}}_{\kappa \kappa }}\left( \zeta  \right) - {\mathscr{Q}}\left( \zeta  \right)\right\|_{{\frac{\sigma}{4},\frac{\sigma}{4}}} &= \mathop {\lim }\limits_{\nu  \to \infty } \left\|{\mathscr{W}}_{yy}^\nu \left( \zeta  \right) - {{\mathscr{Q}}^0}\left( \zeta  \right)\right\|_{{\frac{\sigma}{4},\frac{\sigma}{4}}} \leqslant \mathop {\lim }\limits_{\nu  \to \infty } \sum\limits_{\ell  = 0}^\nu  {\left\|{\mathscr{W}}_{yy}^\ell \left( \zeta  \right) - {{\mathscr{Q}}^\ell }\left( \zeta  \right)\right\|_{{\frac{\sigma}{4},\frac{\sigma}{4}}}} \notag \\
	& \leqslant \mathop {\lim }\limits_{\nu  \to \infty } \sum\limits_{\ell  = 0}^\nu  {\sigma^{-1}{{\rm e}^{ - {2^\ell }K{{\sigma ^{ - \frac{2}{\eta} }}}}}}  \leqslant \sum\limits_{\ell  = 0}^\infty  {\frac{1}{{\sigma {2^\ell }K{{\sigma ^{ - \frac{2}{\eta} }}}}}}  \leqslant \frac{2{\sigma ^{\frac{2}{\eta} -1}}}{{ K}}.\notag
\end{align}
This completes the proof of Theorem \ref{FULLT1}.

\subsection{Proof of Theorem \ref{FULLT2}: KAM via the  infinite-dimensional weak Diophantine non-resonance condition}\label{FULLSEC4}
As we have previously mentioned, the convergence of the  Newton iteration is always to be super-exponential, and this is the essential reason we could generalize the classical infinite-dimensional Diophantine condition in Definition  \ref{FULLDio} to a weaker one, as seen in Definition \ref{weakdio}. The basic framework is similar to the proof of Theorem \ref{FULLT1}, hence we  omit the details here. We also mention the rigorous analysis of the  finite-dimensional version  in \cite{MR4836959}. The key point here is to construct an appropriate contraction sequence and prove the uniform convergence through the KAM process, by employing the boundedness in \eqref{FULLweakcon}, i.e., 
\[\sum\limits_{m = 0}^\infty  {{\delta _m}}  <  + \infty ,\quad \sum\limits_{m = 0}^\infty  {{{\mathscr{E}}^{ - 1}}\big( {{{\rm e}^{{2^m}{\delta _m}}}} \big)}   <  + \infty.\]

Recall \eqref{FULLweakcon} and that $ \sigma>0 $ is sufficiently large. Without loss of generality, let 
\[\sigma  \geqslant 64\sum\limits_{m = 0}^\infty  {{{\mathscr{E}}^{ - 1}}\left( {{{\rm e}^{{2^m}{\delta _m}}}} \right)}.\]
 Then, we construct the desired contraction sequence as
\[{{\widetilde \sigma }_\nu } = \sigma  - 8\sum\limits_{m = 0}^\nu  {{{\mathscr{E}}^{ - 1}}\big( {{{\rm e}^{{2^m}{\delta _m}}}} \big)} ,\quad \nu  \in {\mathbb{N}^ + }.\]
Denote $ {{\mathscr{G}}_\nu }: = C{{\mathscr{E}}^6}\left( {({{\widetilde{\sigma} _\nu } - {\widetilde{\sigma} _{\nu  + 1}}})/8} \right) $,  where  $ C>0 $ is some universal constant. In view of \eqref{FULLweakcon}, one can verify that
\begin{align}
	\sum\limits_{\nu  = 0}^\infty  {\frac{{\log {{\mathscr{G}}_\nu }}}{{{2^\nu }}}}  &= \sum\limits_{\nu  = 0}^\infty  {\frac{1}{{{2^\nu }}}} \left( {\log C + 6\log \left( {{\mathscr{E}}\left( {\frac{{{\widetilde{\sigma} _\nu } - {\widetilde{\sigma} _{\nu  + 1}}}}{8}} \right)} \right)} \right)\notag \\
	& = \sum\limits_{\nu  = 0}^\infty  {\frac{1}{{{2^\nu }}}} \left( {\log C + 6 \cdot {2^{\nu  + 1}}{\delta _{\nu  + 1}}} \right)\notag \\
	\label{FULLlogG}& = 2\log C + 12\sum\limits_{\nu  = 0}^\infty  {{\delta _{\nu  + 1}}}  <  + \infty .
\end{align}

Similar to \eqref{FULLerr1} and \eqref{FULLerr2}, let $ {\widetilde \varepsilon }_\nu $ be the smallest number such that
\begin{align*}
	&{\left\| {{{\mathscr{H}}^\nu }\left( {x,0} \right) - \int_{{\mathbb{T}^\infty }} {{{\mathscr{H}}^\nu }\left( {\xi ,0} \right){\rm d}\xi } } \right\|_{{{\widetilde \sigma }_\nu }}} \leqslant {{\widetilde \varepsilon }_\nu },\\
	&{\left\| {{\mathscr{H}}_y^\nu \left( {x,0} \right) - \omega } \right\|_{{{\widetilde \sigma }_\nu }}} \leqslant {\mathscr{E}}\left( {{{\widetilde \sigma }_\nu } - {{\widetilde \sigma }_{\nu  + 1}}} \right){{\widetilde \varepsilon }_\nu }.
\end{align*}
Then, via a modified KAM iteration, we obtain from \eqref{FULLlogG} that
\begin{align}
	{{\widetilde \varepsilon }_{\nu  + 1}} &\leqslant \left( {\prod\limits_{j = 0}^\nu  {{\mathscr{G}}_{\nu  - j}^{{2^j}}} } \right)\widetilde \varepsilon _0^{{2^{\nu  + 1}}}\notag \\
	& = \exp \left( {\sum\limits_{j = 0}^\nu  {{2^j}\log {{\mathscr{G}}_{\nu  - j}}}  + \left( {2\log {{\widetilde \varepsilon }_0}} \right){2^\nu }} \right)\notag \\
	& = \exp \left( {\left( {\sum\limits_{j = 0}^\nu  {\frac{{\log {{\mathscr{G}}_{\nu  - j}}}}{{{2^{\nu  - j}}}}}  + \left( {2\log {{\widetilde \varepsilon }_0}} \right)} \right){2^\nu }} \right)\notag \\
	& = \exp \left( {\left( {\sum\limits_{j = 0}^\nu  {\frac{{\log {{\mathscr{G}}_j}}}{{{2^j}}}}  + \left( {2\log {{\widetilde \varepsilon }_0}} \right)} \right){2^\nu }} \right)\notag \\
	& \leqslant \exp \left( {\left( {\mathop {\sup }\limits_{\nu  \geqslant 1} \sum\limits_{j = 0}^\nu  {\frac{{\log {{\mathscr{G}}_j}}}{{{2^j}}}}  + \left( {2\log {{\widetilde \varepsilon }_0}} \right)} \right){2^\nu }} \right)\notag \\
	& \leqslant {{\rm e}^{ - {2^{\nu  + 1}}K}},\notag 
\end{align}
 for a universal constant $K$ satisfying
 \[ K\gg - \mathop {\sup }\limits_{\nu  \geqslant 1} \sum\limits_{j = 0}^\nu  {\frac{{\log {{\mathscr{G}}_j}}}{{{2^j}}}}  - 2\log {{\widetilde \varepsilon }_0} > 0 ,\]
 whenever $ {{\widetilde \varepsilon }_0}>0 $ is sufficiently small. Finally, thanks to the super-exponential property $ {{\widetilde \varepsilon }_\nu } \leqslant {{\rm e}^{ - {2^\nu }K}} $ for $ \nu \in \mathbb{N} $, it is straightforward to verify the uniform convergence of the transformation. Therefore, we complete the proof of Theorem \ref{FULLT2}.

\section*{Appendix}

\begin{lemma}\label{lambdapiao}
	Let $ 0<\rho \ll 1 $ and $ \lambda>0 $ be given. Then there exists some $ \widetilde{\lambda} \in (0,1) $ such that
	\[\exp \left( {\frac{x}{{{{\left( {\log \left( {1 + x} \right)} \right)}^{1 + \lambda }}}}} \right) \cdot {{\rm e}^{ - \rho x}} \leqslant \exp \left( {\exp \left( {{\rho ^{ - \widetilde \lambda }}} \right)} \right).\]
\end{lemma}
\begin{proof}
	Note that
	\[\exp \left( {\frac{x}{{{{\left( {\log \left( {1 + x} \right)} \right)}^{1 + \lambda }}}}} \right) \cdot {{\rm e}^{ - \rho x}} = \exp \left( {\frac{x}{{{{\left( {\log \left( {1 + x} \right)} \right)}^{1 + \lambda }}}} - \rho x} \right): = \exp \left( {\varpi \left( x \right)} \right).\]
	Then, it is straightforward to verify that  there exists a unique $ {x^ * } \in \left( {1, + \infty } \right) $, such that $ \varpi \left( {{x^ * }} \right) = \mathop {\max }\nolimits_{x \geqslant 1} \varpi \left( x \right) $, and $ \log {x^ * } \sim {\rho ^{ - \frac{1}{{1 + \lambda }}}},\rho  \to {0^ + } $. We therefore have
	\[\mathop {\max }\limits_{x \geqslant 1} \exp \left( {\varpi \left( x \right)} \right) = \exp \left( {\varpi \left( {{x^ * }} \right)} \right) \leqslant \exp \left( {{x^ * }} \right) \leqslant \exp \left( {\exp \left( {{\rho ^{ - \widetilde \lambda }}} \right)} \right)\]
	with some $ \widetilde{\lambda} \in (0,1) $.	This proves the lemma.
\end{proof}

 \section*{Acknowledgements} 
 We express our profound gratitude to Prof. S. Kuksin and Prof. M. Procesi for providing insightful suggestions that significantly enhanced the quality of this paper. We are also grateful to Prof. L. Warnke for bringing to our attention several recent advances concerning the Lagrange inversion theorem. Z. Tong  was supported by the China Postdoctoral Science Foundation (Grant No. 2025M783102). Y. Li was supported in part by the National Natural Science Foundation of China (Grant Nos. 12071175, 12471183 and 12531009).

\end{document}